\documentclass{amsart}
\usepackage[dvips]{graphicx}
\usepackage{amssymb,amsmath}
\usepackage{epsfig}
\usepackage{color}

\newtheorem{theorem}{Theorem}[section]
\newtheorem{proposition}[theorem]{Proposition}
\newtheorem{lemma}[theorem]{Lemma}
\newtheorem{corollary}[theorem]{Corollary}

\theoremstyle{definition}
\newtheorem{definition}[theorem]{Definition}
\newtheorem{example}[theorem]{Example}
\newtheorem{assumption}[theorem]{Assumption}
\newtheorem{remark}[theorem]{Remark}

\numberwithin{equation}{section}

\newcommand{\Conj}{\mathrm{Conj}}

\begin{document}

\title{Cyclic branched coverings of knots and Quandle homology}
\author{Yuichi Kabaya}
\address{Department of Mathematics, Osaka University, 
Toyonaka, Osaka, 560-0043, JAPAN}
\email{y-kabaya@cr.math.sci.osaka-u.ac.jp}

\subjclass[2000]{57M05; 57M10; 57M12; 57M25; 57M27}
\keywords{quandle homology, group homology, cyclic branched covering}

\begin{abstract}
We give a construction of quandle cocycles from group cocycles,
especially, for any integer $p \geq 3$, quandle cocycles of the dihedral quandle $R_p$ from group cocycles of the cyclic group $\mathbb{Z}/p$.  
We will show that a group 3-cocycle of $\mathbb{Z}/p$ gives rise to a non-trivial quandle 3-cocycle of $R_p$. 
When $p$ is an odd prime, since $\dim_{\mathbb{F}_p} H_Q^3(R_p; \mathbb{F}_p) = 1$, our 3-cocycle is a constant multiple of the Mochizuki 3-cocycle up to coboundary.
Dually, we construct a group cycle represented by a cyclic branched covering branched along a knot $K$ 
from the quandle cycle associated with a colored diagram of $K$. 
\end{abstract}

\maketitle

\section{Introduction}
\label{sed:intro}
A quandle, which was introduced by Joyce \cite{joyce}, is an algebraic object whose axioms are motivated by knot theory and conjugation in a group. 
In \cite{CJKLS}, Carter, Jelsovsky, Kamada, Langford and Saito 
introduced a quandle homology theory, and they defined the quandle cocycle invariants for classical knots and surface knots. 
The quandle homology is defined as the homology of a chain complex generated by cubes whose edges are labeled by elements of a quandle.
On the other hand the group homology is defined as the homology of a chain complex generated by tetrahedra whose edges are labeled by elements of a group.
So it is natural to ask a relation between quandle homology and group homology.
This question also arises from the fact that the quandle cocycle invariants were defined as an analogue of the Dijkgraaf-Witten invariants, 
which are defined using group cocycles.

In \cite{inoue-kabaya}, the authors defined a simplicial version of quandle homology and constructed a homomorphism 
from the usual quandle homology to the simplicial quandle homology.
The important point is that this homomorphism gives a triangulation of a knot complement in algebraic fashion. 
This enables us to relate the quandle homology to the topology of knot complements. 

In this paper, we apply the results of \cite{inoue-kabaya} to construct quandle cocycles from group cocycles.
First, we demonstrate how to give a quandle cocycle of the dihedral quandle $R_p$ from a group cocycle of 
the cyclic group $\mathbb{Z}/p$ for any integer $p \geq 3$ in Section \ref{sec:dihedral}.
We will show that a generator of $H^3(\mathbb{Z}/p; \mathbb{Z}/p)$ gives rise to a non-trivial quandle 3-cocycle of $H^3_Q(R_p; \mathbb{Z}/p)$.
When $p$ is an odd prime, since $\dim_{\mathbb{F}_p} H_Q^3(R_p; \mathbb{F}_p) = 1$, our quandle 3-cocycle is equal to 
a constant multiple of the Mochizuki 3-cocycle \cite{mochizuki} up to coboundary. 

Then we generalize the construction 
to wider classes of quandles. 
Let $G$ be a group and $h$ be an element of $G$, 
then the set $\Conj(h) = \{g^{-1}hg | g \in G\}$ forms a quandle by conjugation.
(It is known by Joyce \cite{joyce} that any faithful homogeneous quandle has such a presentation.)
When some obstruction in second cohomology vanishes, we will construct a quandle cocycle of $\Conj(h)$ from a group cocycle of $G$.

Dually, we relate the quandle cycle associated with an arc and region coloring (shadow coloring) of a knot $K$ 
to a group cycle represented by a cyclic branched covering branched along $K$.
Let $D$ be a diagram of $K$. 
We can define the notion of arc and region colorings of $D$ by a quandle $\Conj(h)$. 
A pair of an arc and a region coloring is called a \emph{shadow coloring}.
We can associate a cycle of a quandle homology group to a shadow coloring of $D$.
Using the homomorphism constructed in \cite{inoue-kabaya}, we will construct a group cycle of $G$ represented by a cyclic branched covering branched along $K$.
This reveals a close relationship between the shadow cocycle invariant of a knot and the Dijkgraaf-Witten invariant of the cyclic branched cover.

We remark that Hatakenaka and Nosaka defined an invariant of 3-manifolds called the \emph{4-fold symmetric quandle homotopy invariant} \cite{hatakenaka-nosaka}, 
based on the fact that any 3-manifold can be represented as a 4-fold simple branched covering of $S^3$ along a link.
As an application, they showed that the shadow cocycle invariant of a link for the Mochizuki $3$-cocycle is equal to 
a scalar multiple of the Dijkgraaf-Witten invariant of the double branched cover along the link. 

This paper is organized as follows.
In Section \ref{sec:group_homology}, we recall the definition of group homology and show how to represent a group cycle by a triangulation with a labeling of its $1$-simplices.
We give a presentation of the fundamental group of a cyclic branched covering branched along a knot in Section \ref{sec:pres_of_cyc_bra_cov}, 
which is independent from the other sections.
In Section \ref{sec:cyc_rep_by_cyc_bra_cov}, we construct a group cycle represented by a cyclic branched cover.
We recall the definition of quandles and their homology theory in Section \ref{sec:quandle_and_homology}.
We review some results from \cite{inoue-kabaya} in Section \ref{sec:main_map} 
and apply them to construct quandle cocycles of the dihedral quandle $R_p$ in Section \ref{sec:dihedral}.
The reader who is interested in the form of the $3$-cocycle, consult (\ref{eq:the_cocycle}) (and (\ref{eq:formula_of_d})).
We will generalize the construction 
to wider classes of quandles in Section \ref{sec:gen_const}.
In Section \ref{sec:quandle_cyc_as_cyc_bra_cov}, 
we construct  a group cycle represented by a cyclic branched covering from the quandle cycle associated with a shadow coloring. 
The reader who is only interested in the construction of quandle cocycles from group cocycles
may skip Sections \ref{sec:group_homology} -- \ref{sec:cyc_rep_by_cyc_bra_cov} except \S \ref{subsec:group_homology}.


\begin{flushleft}
\bf Acknowledgement
\end{flushleft}
The author thanks Takefumi Nosaka for useful discussions.
He also thanks the referees for helpful comments to improve the exposition.
The author is supported by JSPS Research Fellowships for Young Scientists.

\section{Group homology}
\label{sec:group_homology}
In this section, we collect basic facts on group homology.
In \S \ref{subsec:group_homology}, we review the definition of group homology. We refer to \cite{brown} for details.
The material discussed in \S \ref{subsec:cycles_represented_by_triangulations} and \S \ref{subsec:group_cycles_and_representations} was developed in \cite{neumann}.
\subsection{Group homology}
\label{subsec:group_homology}
Let $G$ be a group.
Let $C_n(G)$ be the free $\mathbb{Z}[G]$-module generated by $n$-tuples $[g_1| \dots| g_n]$ of elements of $G$.
Define the boundary map $\partial : C_n(G) \to C_{n-1}(G)$ by
\[
\begin{split}
\partial( [g_1| \dots |g_n] )  =& g_1[g_2|\dots|g_n]  \\
&+ \sum_{i=1}^{n-1}(-1)^i [g_1|\dots|g_ig_{i+1}|\dots|g_n]
+(-1)^n[g_1|\dots|g_{n-1}].
\end{split}
\]
We remark that the chain complex $\{ \dots \to C_1(G) \to C_0(G) \to \mathbb{Z} \to 0 \}$ is acyclic,
where $C_0(G) \cong \mathbb{Z}[G] \to \mathbb{Z}$ is the augmentation map.
So the chain complex $C_*(G)$ gives a free resolution of $\mathbb{Z}$.
Let $M$ be a right $\mathbb{Z}[G]$-module.
The homology of $C_n(G; M) = M \otimes_{\mathbb{Z}[G]} C_n(G)$ is called the \emph{group homology} of $M$ and denoted by $H_n(G; M)$.
In other words, $H_n(G;M) = \mathrm{Tor}_n^{\mathbb{Z}[G]}(M, \mathbb{Z})$.

Let $C'_n(G)$ be the free $\mathbb{Z}$-module generated by $(g_0, \dots, g_n) \in G^{n+1}$. 
Then $C'_n(G)$ is a left $\mathbb{Z}[G]$-module by the action $g (g_0, \dots, g_n) = (g g_0, \dots, g g_n)$.
Define the boundary operator of $C'_n(G)$ by
\[
\partial (g_0, \dots, g_n) = \sum_{i=0}^{n} (-1)^i (g_0, \dots, \widehat{g_i}, \dots ,g_n).
\]
$C_*(G)$ and $C'_*(G)$ are isomorphic as chain complexes.
In fact, the following correspondence gives an isomorphism:
\[
\begin{split}
[g_1| g_2| \dots | g_n] &\leftrightarrow  (1,g_1,g_1g_2, \dots, g_1 \cdots g_n) \\
(\quad  g_0[g_0^{-1}g_1 | g_1^{-1}g_2| \dots | g_{n-1}^{-1} g_n] &\leftrightarrow (g_0,\dots, g_n) \quad)
\end{split}
\]
The notation using $(g_0,\dots, g_n)$ is called \emph{homogeneous} and 
the one using $[g_1|\dots |g_n]$ is called \emph{inhomogeneous}.

Factoring out $C_n(G)$ by the degenerate subcomplex, that is generated by $[g_1|\dots |g_n]$ such that $g_i = 1$ for some $i$,
we obtain the \emph{normalized} chain complex and its homology group.
It is known that the group homology using the normalized chain complex coincides with the homology using the unnormalized one.
In the homogeneous notation, we factor out $C'_n(G)$ by the subcomplex generated by $(g_0, \dots, g_n)$ such that $g_i=g_{i+1}$ for some $i$
to define the normalized chain complex.

For a left $\mathbb{Z}[G]$-module $N$, 
the group cohomology $H^n(G; N)$ is defined as the cohomology of the cochain complex $C^n(G;N) = \mathrm{Hom}_{\mathbb{Z}[G]}(C_n(G),N)$.
Let $A$ be an abelian group.
A cocycle of $C^n(G;A)$ in the homogeneous notation is a function $f: G^{n+1} \to A$ satisfying the following conditions: 
\begin{enumerate}
\item $\displaystyle\sum_{i=0}^{n+1} (-1)^i f(x_0, \dots, \widehat{x_i}, \dots ,x_{n+1}) = 0$, 
\item $f(gx_0, \dots, gx_n) =  f(x_0, \dots, x_n)$ for any $g \in G$. \quad (left invariance)
\end{enumerate}
If $f$ also satisfies
\begin{enumerate}
\item[(3)] $f(x_0, \dots , x_{n} ) = 0$ if $x_i=x_{i+1}$ for some $i$, 
\end{enumerate}
then $f$ is a normalized $n$-cocycle.
We can show that any $n$-cocycle is cohomologous to a normalized $n$-cocycle.

\subsection{Cycles represented by triangulations}
\label{subsec:cycles_represented_by_triangulations}
Let $\Delta$ be an $n$-dimensional simplex. 
We label the vertices of $\Delta$ by $0,1,\dots, n$.
A face of $\Delta$ is presented by a subset of $\{0,1,\dots, n\}$.
Let $\langle i_0, \dots i_k\rangle$ be the face spanned by $i_0, \dots, i_k \in  \{0, \dots , n\}$ with a vertex ordering
given by $i_0, \dots, i_k$.
Any face inherits a vertex ordering from the vertex ordering of $\Delta$
i.e. $\langle i_0, \dots i_k\rangle$ with $i_0 < i_1 < \dots < i_k$.

\begin{definition}
Let $T$ be a CW-complex obtained by gluing a finite number of $n$-dimensional simplices
along their $(n-1)$-dimensional faces in pairs by simplicial homeomorphisms.
We denote the $k$-skeleton of $T$ by $T^{(k)}$.
We assume that the gluing maps preserve the vertex orderings of the faces.
Then $T - T^{(n-3)}$ is homeomorphic to a topological $n$-manifold (not orientable in general).
When $T - T^{(n-3)}$ is oriented, we call $T$ an \emph{ordered $n$-cycle}.
\end{definition}

Consider an $n$-cycle $\sigma$ of $C_n(G; \mathbb{Z})$. 
Then $\sigma$ is represented by a sum 
\[
\sum_{j} \epsilon_j [g_{j 1}| \dots |g_{j n}]
\] 
where $\epsilon_j = \pm 1$ and $g_{j k} \in G$.
For each $[g_{j 1}| \dots | g_{j n}]$, take an $n$-simplex $\Delta_j$.
Then label the edge $\langle i_1 i_2 \rangle$ of $\Delta_j$ by $g_{j i_1} g_{j (i_{1}+1)} \dots g_{j i_2}$ for $i_1 < i_2$.
In particular, we have:
\[
\langle 0,1 \rangle \leftrightarrow g_{j 1} ,  \quad
\langle 1,2 \rangle \leftrightarrow g_{j 2} ,  \quad
\dots \quad,
\langle n-1,n \rangle \leftrightarrow g_{j n}.
\]
We denote the label of $\langle i_1, i_2 \rangle$ by $\lambda \langle i_1, i_2 \rangle$.
For $i_1 > i_2$, label the oriented edge $\langle i_1, i_2 \rangle$ by $\lambda \langle i_2, i_1\rangle ^{-1}$. 
For any 2-dimensional face $\langle i_0, i_1, i_2 \rangle$, 
we have $ \lambda\langle i_0, i_1 \rangle \lambda\langle i_1, i_2 \rangle = \lambda \langle i_0, i_2 \rangle$.
Rewriting them in the homogeneous notation, we assign labels to the vertices of $\Delta_j$ as 
\[
0 \leftrightarrow 1, \quad 1 \leftrightarrow g_{j 1}, \quad 2 \leftrightarrow g_{j 1}  g_{j 2}, \quad \dots \quad, n \leftrightarrow g_{j 1}\dots g_{j n}
\]
up to the left action of $G$ (Figure \ref{fig:labeling}). 
\begin{figure}
\input{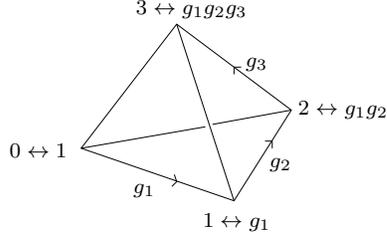}
\caption{A labeling of a simplex.}
\label{fig:labeling}
\end{figure}

Since $\partial \sigma = 0$, $(n-1)$-dimensional faces cancel in pairs.
Gluing $\Delta_j$'s along their faces according to such pairings, we obtain an $n$-cycle $T$. 
At any $(n-1)$-simplex of $T$, there exist exactly two adjacent $n$-simplices.
The labelings of the $(n-1)$-simplex derived from these two $n$-simplices coincide.
Thus we have a well-defined labeling of $1$-simplices $\lambda: \{ \textrm{oriented 1-simplices of $T$} \} \to G$ satisfying,
\begin{enumerate}
\item\label{item:circle} $ \lambda\langle i_0, i_1 \rangle \lambda\langle i_1, i_2 \rangle = \lambda \langle i_0, i_2 \rangle$ 
for any 2-dimensional face $\langle i_0, i_1, i_2 \rangle$, 
\item\label{item:inverse} $ \lambda\langle i_1, i_0 \rangle = \lambda \langle i_0, i_1\rangle^{-1}$.
\end{enumerate}
We call a labeling of $1$-simplices satisfying the conditions (\ref{item:circle}) and (\ref{item:inverse}) \emph{$G$-valued $1$-cocycle}.
Orient $\Delta_j$ positive if $\epsilon_j = 1$ and negative if $\epsilon_j = -1$.
Since these orientations agree on face pairings, thus we have an orientation on $T$.
Therefore $T$ is an ordered $n$-cycle with a $G$-valued $1$-cocycle $\lambda$.
In general, $T$ may not be connected,  but we assume that $T$ is connected because we can treat each connected component separately in our arguments.
Conversely any ordered $n$-cycle $T$ with a $G$-valued $1$-cocycle $\lambda$ represents an $n$-cycle of $C_n(G; \mathbb{Z})$.

\subsection{Group cycles and representations}
\label{subsec:group_cycles_and_representations}
Suppose a cycle $\sigma \in C_n(G; \mathbb{Z})$ is represented by an ordered $n$-cycle $T$ with a $G$-valued $1$-cocycle $\lambda$.
Then $\sigma$ induces a homomorphism from $\pi_1(T)$ to $G$ as follows.
Let $\widetilde{T}$ be the universal covering of $T$ and $p : \widetilde{T} \to T$ be the covering map.
Then the simplices of $T$ lift to simplices of $\widetilde{T}$ and each lift has an induced vertex ordering
compatible with adjacent $n$-simplices.
The $G$-valued $1$-cocycle $\lambda$ of $T$ induces a $G$-valued $1$-cocycle of $\widetilde{T}$.
Consider a ``fundamental domain'' of $T$, that is a contractible subcomplex $D$ of $\widetilde{T}$ 
such that 
\begin{itemize}
\item $\widetilde{T} = \displaystyle\bigcup_{\gamma \in \pi_1(M)} \gamma D$,
\item $D \cap \gamma D = \textrm{(lower dimensional simplices)}$,  for any $\gamma \neq 1$,
\end{itemize}
where we regard $\pi_1(T)$ as the group of deck transformations.
By definition, the number of $n$-simplices in $D$ coincides with the number of $n$-simplices in $T$, in particular finite.
We fix a base point $\widetilde{*}$ in the interior of $D$.
Each $(n-1)$-simplex on $\partial D$ is glued to another $(n-1)$-simplex on $\partial D$.
We denote such pair of faces by $F_i^{\pm}$.
Let $x_i$ be a path in $T$ which starts at $* = p(\widetilde{*})$, traverses $p(F_i)$ in the direction from $F_i^+$ to $F_i^-$ and ends at $*$.
These paths form a system of generators of the fundamental group $\pi_1(T,*)$.
The relations are given at any $(n-2)$-simplices, around which there are a finite number of 
$p(F_i)$'s.

Fix an $n$-simplex $\Delta$ in $D$ and a labeling of vertices of $\Delta$ derived from $\lambda$.
Then $n$-simplices adjacent to $\Delta$ inherit labelings of vertices from $\lambda$.
In this way, all vertices of $D$ are labeled by elements of $G$.
Now consider the labeling of vertices of $F_i^{+}$ and $F_i^{-}$.
Since these reduce to the same labeling of edges, 
they coincide up to left multiplication.
Therefore there exists an element of $G$ which sends the labeling of vertices of $F_i^{-}$ to the one of $F_i^{+}$. 
Denote the element by $\rho(x_i)$. 
This $\rho$ induces a homomorphism $\rho: \pi_1(T,*) \to G$.

Conversely if we have an ordered $n$-cycle $T$ and a homomorphism $\rho: \pi_1(T,*) \to G$, 
we can construct a $G$-valued $1$-cocycle $\lambda$ and then a cycle of $C_n(G; \mathbb{Z})$ up to boundary as follows.
Since $\rho$ induces a map $T \to K(\pi_1(T,*), 1) \to BG$, we obtain a labeling of $1$-simplices $\lambda$ of $T$.
This gives rise to a $G$-valued $1$-cocycle $\lambda$ and a homology class in $H_n(G;\mathbb{Z})$.
The $G$-valued $1$-cocycle $\lambda$ is well-defined up to the coboundary action. 
A map $\mu: \{ \textrm{0-simplices of $T$} \}  \to G$ acts on a $G$-valued $1$-cocycle $\lambda$ as a coboundary action by
\[
\langle i_1,i_2 \rangle \mapsto \mu(i_1)^{-1} \lambda \langle i_1, i_2 \rangle \mu(i_2).
\]
We can show that the homology class obtained from $\lambda$ does not change under the coboundary action. 
For any $g \in G$, the cycle corresponding to the representation $g^{-1} \rho g$ is obtained from $\lambda$ by the coboundary action by $\mu \equiv  g$.
As a result, the homology class obtained from $\rho$ depends only on the conjugacy class of $\rho$.

For a closed oriented $n$-manifold $M$ and a representation $\rho : \pi_1(M) \to G$, 
we have a homology class defined by the image of the fundamental class $[M]$ under the map $H_n(M) \to H_n(K(\pi_1(M),1)) \to H_n(G;\mathbb{Z})$.
When $M$ is homeomorphic to an ordered $n$-cycle $T$, the homology class is represented by a $G$-valued $1$-cocycle $\lambda$ of $T$ associated with $\rho$.
In this situation, we say that the homology class defined by $M$ and $\rho$ is \emph{represented} by $T$ and $\lambda$.

\section{Cyclic branched covering}
\label{sec:pres_of_cyc_bra_cov}
In this section, we give a presentation of the fundamental group of a cyclic branched covering from the Wirtinger presentation.
\subsection{Presentation of the fundamental group of the branched cover}
Let $K$ be a knot in $S^3$ and $D$ be a diagram of $K$.
Then $\pi_1(S^3 - K)$ is presented by generators and relations, called the Wirtinger presentation.
Let $x_1, \dots , x_n$ be the generators of the Wirtinger presentation, that correspond to the arcs of $D$.
Each crossing (Figure \ref{fig:wirtinger_relation}) gives rise to a relation $x_k = x_j^{-1} x_i x_j$.
\begin{figure}
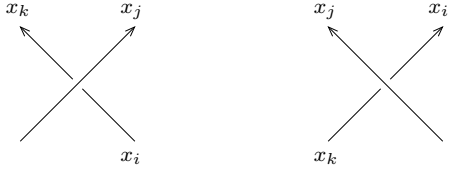

\input{wirtinger_relation_1.pstex_t}
\hspace{50pt}
\input{wirtinger_relation_2.pstex_t}
\caption{The relation is given by $x_k = x_j^{-1} x_i x_j$.}
\label{fig:wirtinger_relation}
\end{figure}

For any integer $l > 1$, let $C_l$ be the $l$-fold cyclic covering of $K$, i.e. the manifold corresponding to the kernel of 
\[
\pi_1(S^3 - K) \to H_1( S^3 - K ) \cong \mathbb{Z} \to \mathbb{Z}/l.
\]
Putting back the knot $K$ to $C_l$, we obtain the $l$-fold cyclic branched covering $\widehat{C}_l$ of $K$.
\begin{proposition}
\label{prop:presentation_of_pi_one}
$\pi_1(C_l)$ has the following presentation:
\[
\begin{split}
\textrm{Generators: }& x_{i,s} \quad (i=1,2, \dots , n, \quad s = 0,1, \dots , l-1), \\
\textrm{Relations: }& x_{k,s} = x_{j,s-1}^{-1} x_{i,s-1} x_{j,s}  \quad (\textrm{for each crossing and } s=0,1, \dots, l-1), \\ 
& x_{1,0}=x_{1,1}= \dots =x_{1,l-2}=1 .
\end{split}
\]
The inclusion map $\pi_1(C_l) \to \pi_1(S^3 - K)$ is given by
\[
x_{i,s} \mapsto x_1^{s-1} x_i x_1^{-s}, 
\]
if we take appropriate base points.
By adding a relation $x_{1,l-1}=1$, we obtain a presentation of $\pi_1(\widehat{C}_l)$.
\end{proposition}
A method for obtaining a presentation of the fundamental group of a branched covering is given in \cite{rolfsen}.
But we give a proof here because some techniques will be used to construct a group cycle represented by $\widehat{C}_l$ later.
\begin{proof}
First, we construct a handle decomposition of the knot complement associated with the Wirtinger presentation (Figure \ref{fig:trefoil}).
Then we lift the handle decomposition to a handle decomposition of $C_l$ (Figure \ref{fig:lifted_splitting}).
After that, we read the relations given by attaching 2-handles.

Let $N(K)$ be a regular neighborhood of $K$.
We shall give a handle decomposition of $S^3 - N(K)$.
We represent $S^3$ by the one point compactification of $\mathbb{R}^3=\{(x,y,z) |  x,y,z \in \mathbb{R} \}$.
Let $B_+ = \{ z \geq 0\} \cup \{ \infty \}$ and $B_- = \{ z \leq 0 \} \cup \{ \infty \}$.
We denote the equatorial sphere of $S^3$ by $S_0= \{ z=0 \} \cup \{ \infty \}$.
Put $K$ in a position such that the projection to $S_0$ has only double points.
Let $n$ be the number of the crossings of this projection.
We deform $K$ in the $z$-direction so that $K$ intersects $S_0$ at $2n$ points and each of $B_{\pm} \cap K$ consists of $n$ arcs.
We call $B_+ \cap K$ over-crossing arcs and $B_{-} \cap K$ under-crossing arcs.
Index the arcs of $B_{+} \cap K$ by $x_i$ $(i=1,2, \dots, n)$.

Now $B_{+} - N(K)$ is homeomorphic to a handlebody of genus $n$ (see Figure \ref{fig:trefoil}). 
Projecting the over-crossing arcs to $S_0$, we obtain a meridian disk system of the handlebody.
We denote the meridian disk corresponding to $x_i$ by $D_i$.
For each under-crossing arc, attach a 2-handle $D^2 \times  D^1$ along $\partial D^2 \times D^1$ to $B_+ - N(K)$ (see Figure \ref{fig:trefoil}). 
Then the resulting manifold is homeomorphic to $(S^3 - N(K))-B^3$ where $B^3$ is a 3-ball.
Attaching $B^3$ along the boundary to $\partial ((S^3 - N(K)) - B^3 )$, the resulting manifold is homeomorphic to $S^3 - N(K)$.
So we  have a handle decomposition of $S^3 - N(K)$ into one 0-handle, $n$ 1-handles, $n$ 2-handles and one 3-handle. 
In this handle decomposition, 1-handles correspond to the Wirtinger generators and 2-handles to the Wirtinger relations.
We denote the set of $i$-handles by $h^{i}$ and $X^{(i)} = h^{0} \cup \dots \cup h^{i}$.

\begin{figure}
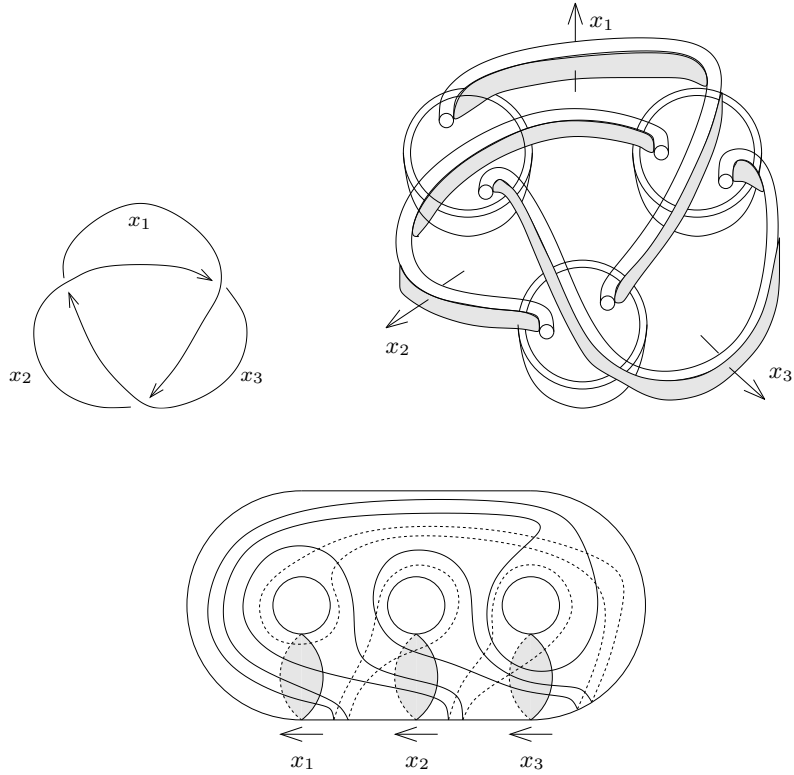

\input{trefoil.pstex_t}
\hspace{30pt}
\input{heegaard.pstex_t}

\vspace{30pt}
\input{heegaard2.pstex_t}
\caption{A handle decomposition of a knot complement.}
\label{fig:trefoil}
\end{figure}

Next we consider the preimage of $X^{(1)}$ in $C_l$.
Cut $X^{(1)}$ along the meridian disks $D_i$, and denote the resulting manifold by $B$, which is homeomorphic to a 3-ball.
Let $*$ be a point in $B \subset S^3 - N(K)$.
We take a loop in $X^{(1)}$ which starts at $*$, then intersects the meridian disk $D_i$ and ends at $*$.
Orient the loop so that it corresponds to the generator of $H_1(S^3 -K)$.
By abuse of notation, we also denote this loop by $x_i$. 
Take a lift $B_0 \subset C_l$ of $B$ and denote the preimage of $* \in B$ in $B_0$ by $\tilde{*}$.
Then there exists a unique lift $\tilde{x}_i$ of the loop $x_i$ starting at $\tilde{*}$.
Since $x_i$ corresponds to the generator of $H_1(S^3 - K)$, $\tilde{x}_i$ ends at another lift of $B$,  we denote it by $B_1$.   
Similarly the lift of $x_i$ starting at $B_1$ ends at another lift, we denote it by $B_2$.
Continuing this, we see that $B_l = B_0$ and all lifts of $B$ will appear.
Therefore the preimage of $X^{(1)}$ is decomposed into $l$ 3-balls  $B_0 \cup B_1 \cup \dots \cup B_{l-1}$.
The intersection of $B_s$ and $B_{s+1}$ consists of $n$ disks each of which is a lift of a meridian disk $D_i$. 
We denote this lifted disk by $D_{i,s}$ (Figure \ref{fig:lifted_splitting}). 
It is easy to check that $ \{D_{1,l-1}\} \cup  \{D_{i,s} \}_{i=2, \dots, n , s=0, \dots l-1}$ forms a meridian disk system of the preimage of $X^{(1)}$. 
Denote the lift of $x_i$ starting at $B_s$ (ending at $B_{s+1}$) by $\widetilde{x}_{i,s}$.
Then the generator of $\pi_1(C_l, \widetilde{*})$ corresponding to the meridian disk $D_{i,s}$ is given by
\begin{equation}
\label{eq:generator_in_cyclic_cover}
\widetilde{x}_{1,0}\widetilde{x}_{1,1} \dots \widetilde{x}_{1,s-1} 
\widetilde{x}_{i,s}
\widetilde{x}_{1,s}^{-1}\widetilde{x}_{1,s-1}^{-1} \dots \widetilde{x}_{1,0}^{-1} .
\end{equation}
We denote this element by $x_{i,s}$.
To give a simple presentation of $\pi_1(C_l, \widetilde{*})$,  
we add extra generators $x_{1,0}, x_{1,1}, \dots, x_{1,l-2}$ corresponding to $\widetilde{x}_{1,0}, \widetilde{x}_{1,1}, \dots \widetilde{x}_{1,l-2}$ respectively
and relations $x_{1,0} = x_{1,1} = \cdots = x_{1,l-2} = 1$.

Finally we consider the relations given by the lifts of 2-handles.
We see that the relation $x_k = x_j^{-1} x_i x_j$ lifts to 
\[
x_{k,s} = x_{j,s-1}^{-1} x_{i,s-1} x_{j,s} \quad (s=0,1,\dots, l-1), 
\]
see Figure \ref{fig:lifted_splitting}.

Since the generator $x_{i,s}$ is represented by (\ref{eq:generator_in_cyclic_cover}), 
the inclusion map $\pi_1(C_l. \tilde{*}) \to \pi_1(S^3-K, *)$ is given by $x_{i,s} \mapsto x_1^{s-1} x_s  x_1^{-s}$.
This proves the second statement.

By adding a 2-handle to $C_l$ along $\widetilde{x}_{1,0}\widetilde{x}_{1,1} \dots \widetilde{x}_{1,l-1}$ and capping off the resulting sphere, 
we obtain a manifold homeomorphic to the cyclic branched covering $\widehat{C}_l$.
Therefore a presentation of the cyclic branched covering is obtained by adding a relation $x_{1,l-1}=1$.
\begin{figure}
\input{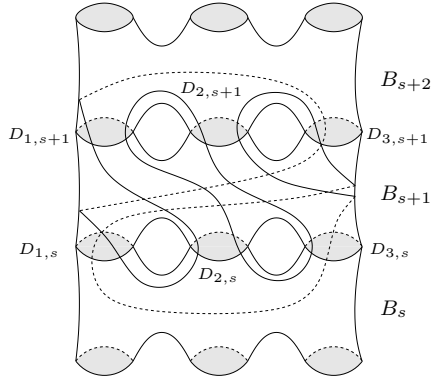}
\caption{A handle decomposition of $C_l$.}
\label{fig:lifted_splitting}
\end{figure}
\end{proof}

\section{Cycle represented by cyclic branched covering}
\label{sec:cyc_rep_by_cyc_bra_cov}
For a representation $\rho : \pi_1(S^3 - K) \to G$, we have the restriction map 
$\rho|_{\pi_1(C_l)} : \pi_1(C_l) \to G$ given by
\begin{equation}
\label{eq:covering_representation}
\rho|_{\pi_1(C_l)}( x_{i,s} ) = \rho(x_1)^{s-1} \rho(x_i) \rho(x_1)^{-s}.
\end{equation}
If $\rho(x_1)^l = 1$, it reduces to a representation $\widehat{\rho} :  \pi_1(\widehat{C}_l) \to G$
and there is a group cycle given by $\widehat{C_l}$ and $\widehat{\rho}$.
In this section, we construct an explicit ordered $3$-cycle and its $G$-valued $1$-cocycle representing the homology class given by $\widehat{C_l}$ and $\widehat{\rho}$.
First we give a triangulation of $S^3 - N(K)$ associated with the Wirtinger presentation.

\begin{figure}
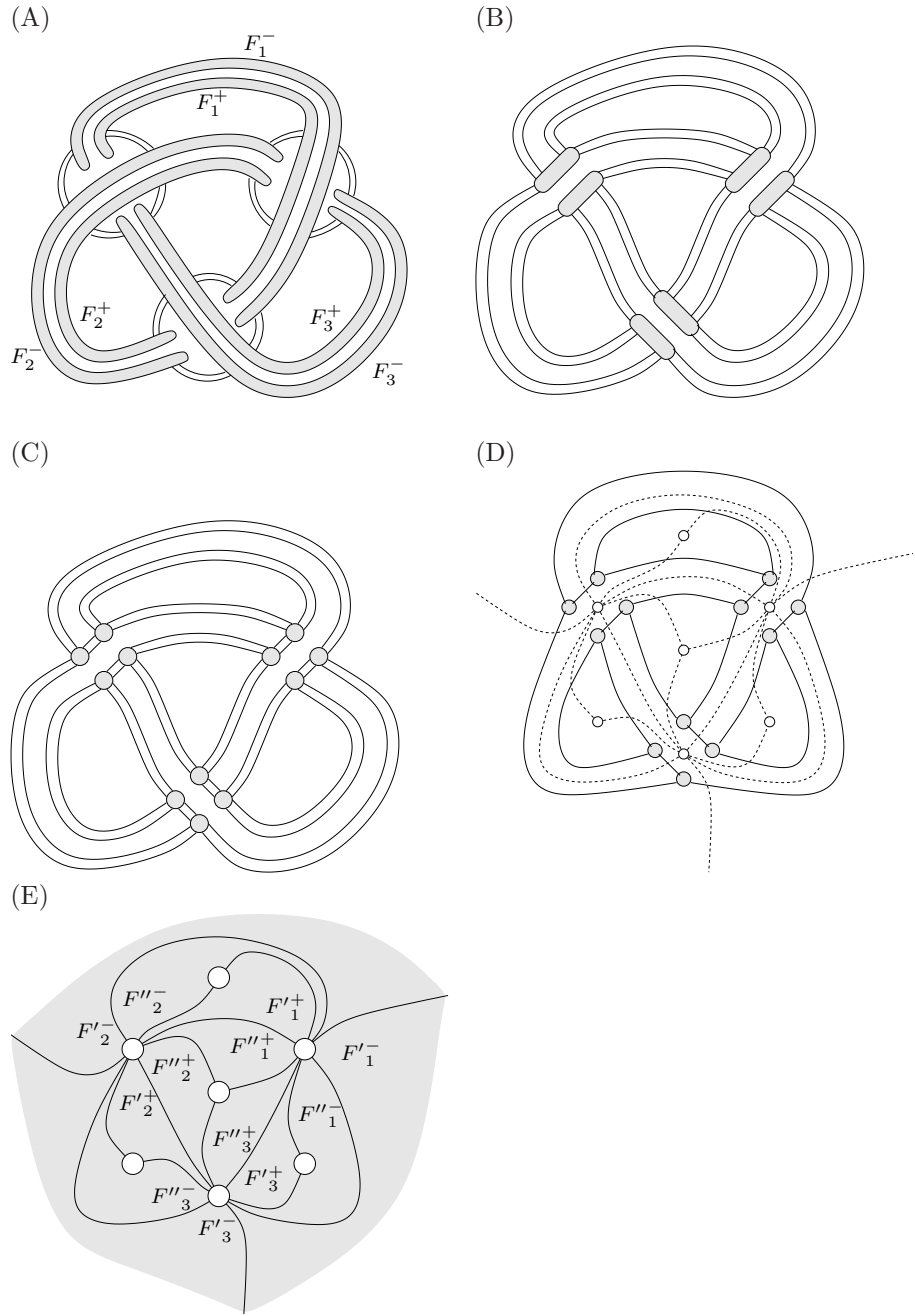


\begin{tabular}{ll}
(A) & (B) \\
\input{cut_heegaard.pstex_t} 
& 
\input{cut_heegaard2.pstex_t} \\
& \\
(C) & (D) \\
\input{stabilized.pstex_t}
&
\input{dual_graph.pstex_t} \\
(E) & \\
\input{triangulation.pstex_t} 
\end{tabular}
\caption{A polyhedral decomposition of $S^3-N(K)$.
We are viewing from inside of $B_+$.
The white rectangles in (A) and (B) correspond to the attaching regions of the 2-handles.}
\label{fig:cut_heegaard}
\end{figure}

Let $K$ be a knot and fix a diagram of $K$.
As in the proof of Proposition \ref{prop:presentation_of_pi_one}, 
we define $B_{\pm}$ and give a Heegaard splitting of $B_+ - N(K)$ with the meridian disks $D_i$.
Cutting the handlebody $B_+ - N(K)$ along the meridian disks $D_i$, the result is a ball with $2n$ 2-cells on the boundary.
We denote the resulting 3-ball by $B$ and the pair of 2-cells corresponding to $D_i$ by $F_i^{+}$ and $F_i^-$  
so that the Wirtinger generator corresponding to $x_i$ runs from $F_i^{+}$ to $F_i^{-}$ (Figure \ref{fig:cut_heegaard}(A)). 
Now consider the attaching regions of the two handles corresponding to under-crossing arcs, 
which consist of annuli  $S^1 \times D^1 (\subset D^2 \times D^1)$ on $\partial (B_+ - N(K))$.
Each annulus 
is divided by $D_i$'s into 4 rectangles on $\partial B$. 
We define a graph on $\partial B$ with vertices consisting of $F_i^{\pm}$ and edges consisting of these rectangles (Figure \ref{fig:cut_heegaard}(B)).
Each vertex of this graph has valency at least four.  
We can make all vertices of the graph into trivalent vertices (Figure \ref{fig:cut_heegaard}(C)) by adding extra $1$-handles and $2$-handles 
(stabilizations of the Heegaard splitting, see Figure \ref{fig:stabilization}). 
We denote the new vertices by ${F'}_i^{\pm}$, ${F''}_i^{\pm}$, ... which originally belonged to $F_i^{\pm}$.
The dual of the graph gives a triangulation of $\partial B$ (Figures \ref{fig:cut_heegaard}(D) and \ref{fig:cut_heegaard}(E)).
By abuse of notation, we denote the triangles dual to the vertices ${F'}_i^{\pm}$, ${F''}_i^{\pm}$, ... by the same symbols. 
Taking a cone from an interior point of $B$, we obtain a triangulation $T$ of $B$ into $4n$ tetrahedra.
Regluing the triangles ${F'}_i^{+}$, ${F''}_i^{+}$, ... to ${F'}_i^{-}$, ${F''}_i^{-}$, ..., 
we obtain a triangulation of $S^3 - N(K)$ into $4n$ tetrahedra, which was explained in \cite{weeks}.
We remark that this is not a triangulation in the usual sense: it is not a simplicial complex, moreover the link of some $0$-simplex is not homeomorphic to the $2$-sphere.
Actually there exists only three $0$-simplices, one is the cone point in $B$ (north pole in \cite{weeks}), 
second is the $0$-simplex corresponding to the complementary regions of the diagram (south pole), 
the last one is a $0$-simplex whose small neighborhood is homeomorphic to the cone over the torus $\partial N(K)$.

\begin{figure}
\input{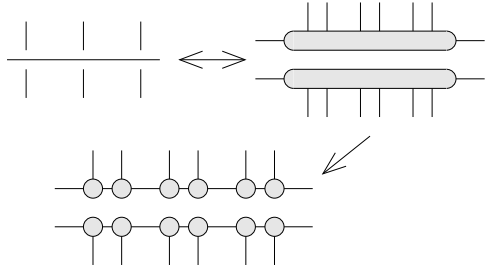}
\caption{Stabilizations of the handle decomposition.}
\label{fig:stabilization}
\end{figure}

\begin{figure}
\input{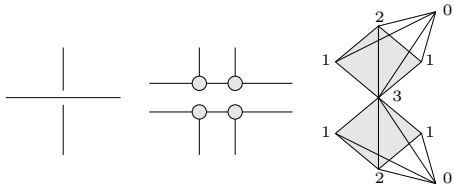}
\caption{The orderings of the tetrahedra of $B$.}
\label{fig:ordering}
\end{figure}

Using this triangulation, we construct a triangulation of the cyclic branched covering $\widehat{C}_l$.
Let $B_0, B_1, \dots, B_{l-1}$ be $l$ copies of $B$ and $T_s$ be the triangulation of $B_s$ we have constructed.
We denote the triangles ${F'}_i^{\pm}$, ${F''}_i^{\pm}$, ...  on $\partial B_s$ by  ${F'}_{i,s}^{\pm}$, ${F''}_{i,s}^{\pm}$, ... respectively.
By abuse of notation, we regard ${F'}_{i,s}^{\pm}$, ${F''}_{i,s}^{\pm}$, ...  simply as $F_{i,s}^{\pm}$.
Glue $T_s$'s along their boundary triangles by the following identification maps
\[
F_{i,s}^- \to F_{i,s-1}^+ \quad (i=1,2, \dots ,n, \quad s = 0,1, \dots, l-1 ).
\]
Denote this triangulation by $\widehat{T}$.
We define an ordering of each tetrahedron by assigning $0$ to the interior vertex of $B_s$, $1$ to the vertex corresponding to complementary region of the diagram, 
$2$ to the vertex corresponding to under-crossing arc, $3$ to the vertex corresponding to over-crossing arc, respectively (Figure \ref{fig:ordering}). 
The orderings are compatible under the gluing maps.  
So $\widehat{T}$ is an ordered $3$-cycle.
Here $\widehat{T}$ is a triangulation of $\widehat{C}_l$ except near the $0$-simplex corresponding to $K$, 
whose small neighborhood is homeomorphic to a cone over a torus.
We can resolve this singularity by inserting a suspension of a $2l$-gon around $K$ (Figure \ref{fig:prism}).
As a result, we obtain an ordered $3$-cycle homeomorphic to $\widehat{C}_l$.
(This procedure is called ``blowing up'' at a $0$-simplex in \cite{neumann}.)

\begin{figure}
\input{prism.pstex_t}
\caption{}
\label{fig:prism}
\end{figure}

We construct a group cycle given by $\widehat{C}_l$ and a representation $\widehat{\rho}:\pi_1(\widehat{C}_l) \to G$ using the triangulation $\widehat{T}$.
Here $\widehat{\rho}$ is given by the set $\{ \widehat{\rho}(x_{i,s}) \} \subset G $ satisfying 
\[
\begin{split}
\widehat{\rho}(x_{k,s}) &= \widehat{\rho}(x_{j,s-1})^{-1} \widehat{\rho}(x_{i,s-1}) \widehat{\rho}(x_{j,s})  \quad (\textrm{$i,j,k$ as in Figure \ref{fig:wirtinger_relation}}, 
\quad s=0,1, \dots, l-1), \\ 
\widehat{\rho}(x_{1,0}) &= \widehat{\rho}(x_{1,1}) = \widehat{\rho}(x_{1,2}) = \dots = \widehat{\rho}(x_{1,l-1}) =1. \\
\end{split}
\]
Give a labeling of vertices on $T_s$ for each $s$.
Let $(g_1,g_2,g_3)$ be the labeling of vertices on $F_{i,s+1}^-$ and $(g'_1,g'_2,g'_3)$ be the labeling of vertices on $F_{i,s}^+$.
If these are related by 
\[
\widehat{\rho}(x_{i,s})(g_1,g_2,g_3) = (g'_1,g'_2, g'_3)
\]
then we obtain a $G$-valued $1$-cocycle $\lambda$ on $\widehat{T}$ by gluing $T_s$ using $\widehat{\rho}$.
To obtain a group cycle given by $\widehat{C}_l$ and $\widehat{\rho}$, we insert a suspension over a $2l$-gon to $\widehat{T}$ at the $0$-simplex corresponding to $K$.
We give an ordering on the vertices of each tetrahedron of the suspension compatible with the ordering on the $2l$-gon, 
e.g. order the central $0$-simplex maximal. 
These tetrahedra inherit a vertex labeling by $G$ on the boundary faces of the suspension. 
We assign any labeling at the central $0$-simplex of the suspension.
Then upper and lower tetrahedra have the same labelings with different orientations.
So these cancel out in pairs and give no contribution to the group cycle represented by $\widehat{C}_l$ and $\widehat{\rho}$.
Therefore the homology class given by $\widehat{T}$ and $\lambda$ represents the homology class given by $\widehat{C}_l$ and $\widehat{\rho}$.

\section{Quandle homology}
\label{sec:quandle_and_homology}
In this section, we review the definitions of quandles, rack (co)homology and quandle (co)homology.
Our treatment of quandle (or rack) homology follows that of \cite{etingof-grana}.
In \S \ref{subsec:shadow} and \S \ref{subsec:cocycle_invariant}, we recall the notions of colorings and quandle cocycle invariants defined in \cite{CJKLS}. 
\subsection{Quandle and quandle homology}
A \emph{quandle} $X$ is a set with a binary operation $*$ satisfying the following axioms:
\begin{itemize}
\item[(Q1)] $x*x = x$ for any $x \in X$, 
\item[(Q2)] the map $*y:X \to X$ defined by $x \mapsto x*y$ is a bijection for any $y \in X$, 
\item[(Q3)] $(x*y)*z = (x*z) * (y*z) $ for any $x,y,z \in X$. 
\end{itemize}
We denote the inverse of $*y$ by $*^{-1} y$.
For a quandle $X$, we define the \emph{associated group} $G_X$ by 
\[
G_X = \langle x \in X | y^{-1} x y = x*y  \quad (x,y \in X) \rangle.
\]
A quandle $X$ has a right $G_X$-action in the following way.
Let $g=x_1^{\epsilon_1} x_2^{\epsilon_2} \cdots x_n^{\epsilon_n}$ be an element of $G_X$ where $x_i \in X$ and $\epsilon_i = \pm 1$.
Define $x*g = ( \cdots ((x*^{\epsilon_1}x_1) *^{\epsilon_2}x_2 ) \cdots )*^{\epsilon_n}x_n $.
One can easily check that this is a right action of $G_X$ on $X$.
So the free abelian group $\mathbb{Z}[X]$ generated by $X$ is a right $\mathbb{Z}[G_X]$-module.

Let $C^R_n(X)$ be the free left $\mathbb{Z}[G_X]$-module generated by $X^n$.
We define the boundary map $C^R_n(X) \to C^R_{n-1}(X)$ by 
\[
\begin{split}
\partial (x_1, x_2, \dots, x_n) = 
\sum_{i=1}^n (-1)^{i}( &(x_1,\dots,\widehat{x_i},\dots,x_n ) \\
& - x_i(x_1*x_i, \dots, x_{i-1}*x_i, x_{i+1}, \dots, x_n)).
\end{split}
\]
Figure \ref{fig:boundary_map} shows a graphical picture of the boundary map.
\begin{figure}
\input{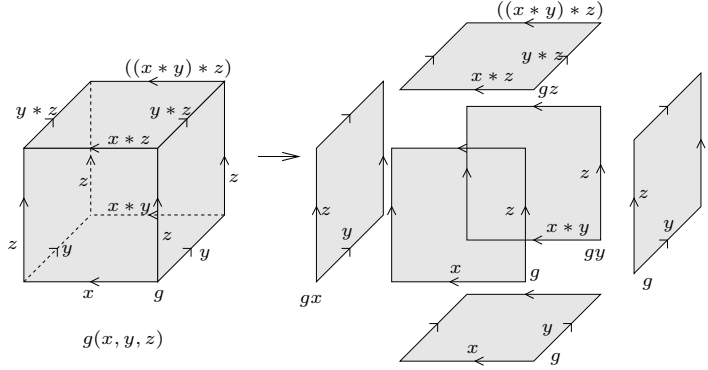}
\caption{$\partial(g (x,y,z)) = - (g (y,z) - gx(y,z)) + (g(x,z)- gy(x*y,z)) - (g(x,y) -gz(x*z,y*z))$.
Here $x,y,z \in X$ and $g \in G_X$. Edges are labeled by elements of $X$ and vertices are labeled by elements of $G_X$.}
\label{fig:boundary_map}
\end{figure}
Let $C^D_n(X)$ be the $\mathbb{Z}[G_X]$-submodule of $C^R_n(X)$ generated by $(x_1,\dots, x_n)$ with $x_i=x_{i+1}$ for some $i$.
Now $C^D_n(X)$ is a subcomplex of $C^R_n(X)$.
Let $C^Q_n(X) = C^R_n(X) / C^D_n(X)$.
For a right $\mathbb{Z}[G_X]$-module $M$,
we define the \emph{rack homology} of $M$ by the homology of $C^R_n(X;M) = M \otimes_{\mathbb{Z}[G_X]} C^R_n(X)$
and denote it by $H_n^R(X;M)$.
We also define the \emph{quandle homology} of $M$ by the homology of $C^Q_n(X;M) = M \otimes_{\mathbb{Z}[G_X]} C^Q_n(X)$
and denote it by $H_n^Q(X;M)$.
The homology $H_n^Q(X;\mathbb{Z})$, here $\mathbb{Z}$ is the trivial $\mathbb{Z}[G_X]$-module, is equal to the usual quandle homology $H^Q_n(X)$.
Let $Y$ be a set with a right $G_X$-action.
For any abelian group $A$, the abelian group $A[Y]$ freely generated by $Y$ over $A$ is a right $\mathbb{Z}[G_X]$-module.
The homology group $H^Q_n(X;A[Y])$ is usually denoted by $H^Q_n(X; A)_Y$ (\cite{kamada}).

Let $N$ be a left $\mathbb{Z}[G_X]$-module.
We define the \emph{rack cohomology} $H_R^n(X; N)$ by the cohomology of $C_R^n(X;N) =\mathrm{Hom}_{\mathbb{Z}[G_X]} (C^R_n(X), N)$.
The \emph{quandle cohomology} $H_Q^n(X;N)$ is defined in a similar way.
For a set $Y$ with a right $G_X$-action and an abelian group $A$, we let 
$\mathrm{Func}(Y,A)$ be the left $\mathbb{Z}[G_X]$-module generated by functions $\phi :Y \to A$,
where the action is defined by $(g \phi) (y) = \phi(y g)$ for $y \in Y $ and $g \in G_X$.
The cohomology group $H_Q^n(X; \mathrm{Func}(Y,A))$ is usually denoted by $H_Q^n(X; A)_Y$ (\cite{kamada}).

\subsection{Shadow coloring and associated quandle cycle}
\label{subsec:shadow}
Let $X$ be a quandle.
Let $L$ be an oriented link in $S^3$ and $D$ be a diagram of $L$.
An \emph{arc coloring} of $D$ is an assignment of elements of $X$ to arcs of $D$ satisfying the following relation at each crossing, 
\begin{center}
\begin{minipage}{55pt}
\input{arc_coloring.pstex_t} 
\end{minipage}
\end{center}
where $x, y \in X$. 
By the Wirtinger presentation of the knot complement, an arc coloring determines a representation of $\pi_1(S^3 - L)$  into the associated group $G_X$.
This is obtained by sending each meridian to its color.

Let $Y$ be a set with a right $G_X$ action.
A \emph{region coloring} of $D$ is an assignment of elements of $Y$ to regions of $D$ satisfying the relation
\begin{center}
\begin{minipage}{55pt}

\vspace{10pt}
\input{reg_coloring.pstex_t}

\vspace{10pt}
\end{minipage}
\end{center}
for any pair of adjacent regions, where $r \in Y$ and $x \in X$.
A pair $\mathcal{S}=(\mathcal{A}, \mathcal{R})$ is called a \emph{shadow coloring}.
If we fix a color of a region of $D$, then colors of other regions are uniquely determined.
Therefore there always exists a region coloring compatible with a given arc coloring.

We define a cycle $C(\mathcal{S})$ of $C^Q_2(X; \mathbb{Z}[Y])$ for a shadow coloring $\mathcal{S}$.
Assign $+ r \otimes (x,y)$ for a positive crossing 
and $- r \otimes (x,y)$ for a negative crossing colored by
\begin{center}
\begin{minipage}{50pt}
\input{pos_gen.pstex_t}
\end{minipage}
\hspace{30pt}
\begin{minipage}{50pt}
\input{neg_gen.pstex_t}
\end{minipage}
\end{center}
respectively.
Then define
\[
C(\mathcal{S}) = \sum_{c: \textrm{crossing}} \epsilon_c r_c \otimes (x_c, y_c) \in C^Q_2(X; \mathbb{Z}[Y]),
\]
where $\epsilon_c=\pm 1$. 
We can easily check that 
\begin{proposition}[see \cite{inoue-kabaya}]
\label{prop:not_depend_on_region_coloring}
$C(\mathcal{S})$ is a cycle and the homology class $[C(\mathcal{S})]$ 
is invariant under Reidemeister moves.
Moreover it does not depend on the choice of the region coloring if the action of $G_X$ on $Y$ is transitive.
\end{proposition}

So the homology class $[C(\mathcal{S})]$ is an invariant of the arc coloring $\mathcal{A}$ in many cases.
There are two important sets with right $G_X$-action: one is when $Y$ consists of one point $\{*\}$ and the other is when $Y=X$.
Eisermann showed that the cycle $[C(\mathcal{S})]$ for $Y=\{*\}$ is essentially described by the monodromy of 
some representation of the knot group along the longitude (\cite{eisermann}, \cite{eisermann_coloring}).
So we concentrate on the invariant $[C(\mathcal{S})]$ in the case of $Y=X$ from now on.

\subsection{Quandle cocycle invariant}
\label{subsec:cocycle_invariant}
Let $X$ be a quandle with $ |X| < \infty$.
Let $A$ be an abelian group and  $f$ be a cocycle of $H_Q^2(X; \mathrm{Func}(X,A))$. 
We define the \emph{(shadow) quandle cocycle invariant} by
\[
\frac{1}{|X|} \sum_{\mathcal{S} : \textrm{shadow colorings}}\langle f, C(\mathcal{S})\rangle \in \mathbb{Z}[A].
\]
Here the sum is finite because there are only a finite number of shadow colorings of $D$.
This is an invariant of oriented knots by Proposition \ref{prop:not_depend_on_region_coloring}. 
If $G_X$ acts on $X$ transitively (i.e. $X$ is connected), $\langle f, C(\mathcal{S})\rangle$ does not depend on the choice of a region coloring $\mathcal{R}$ 
by Proposition \ref{prop:not_depend_on_region_coloring}, thus  
\[
\sum_{\tiny \begin{split} \mathcal{S} &= (\mathcal{A}, \mathcal{R}), \\ \mathcal{A} : & \textrm{arc coloring} \end{split}}\langle f, C(\mathcal{S})\rangle 
\]
coincides with the quandle cocycle invariant, 
where $\mathcal{R}$ is a region coloring compatible with $\mathcal{A}$.

We can regard $f \in C_Q^n(X ; A)$ as an element of $C_Q^{n-1}(X; \mathrm{Func}(X,A))$ by
\[
f(x_1, x_2, \dots, x_{n-1}) (r) = f(r,x_1,x_2,\dots,x_{n-1}).
\]
This gives a homomorphism $H_Q^n(X; A) \to H_Q^{n-1}(X; \mathrm{Func}(X,A))$.
Therefore a quandle $3$-cocycle $f \in H_Q^3(X; A)$ gives rise to a quandle cocycle invariant.
Explicitly, the cocycle invariant has the following form:
\[
\sum_{\tiny \begin{split} \mathcal{S} &= (\mathcal{A}, \mathcal{R}), \\ \mathcal{A} : & \textrm{arc coloring} \end{split}} 
\sum_{c: \mathrm{crossing}} \epsilon_c f(r_c, x_c, y_c) \in \mathbb{Z}[A]
\]
where $x_c, y_c \in X$ are given by $\mathcal{A}$ and $r_c \in X$ are given by $\mathcal{R}$.

\section{$H^{\Delta}_n(X; \mathbb{Z})$ and the map $H^R_n (X; \mathbb{Z}[X]) \to H^{\Delta}_{n+1}(X;\mathbb{Z})$}
\label{sec:main_map}
Let $X$ be a quandle.
Let $C^{\Delta}_n(X) = \mathrm{span}_{\mathbb{Z}} \{ (x_0, \dots , x_n) | x_i \in X \}$. 
We define the boundary operator of $C^{\Delta}_n(X)$ by 
\[
\partial (x_0,\dots, x_n)= \sum_{i=0}^n (-1)^i(x_0,\dots , \widehat{x_i}, \dots, x_n).
\]
Since $X$ has a right action of $G_X$, the chain complex $C^{\Delta}_n(X)$ has a right action of $G_X$ by $(x_0,\dots, x_n) * g = (x_0*g,\dots, x_n*g)$.
Let $M$ be a left $\mathbb{Z}[G_X]$-module.
We denote the homology of $C^{\Delta}_n(X) \otimes_{\mathbb{Z}[G_X]} M$ by $H^{\Delta}_n(X;M)$. 
For any abelian group $A$, we can also define the cohomology group $H_{\Delta}^n(X;A)$ in a similar way. 

Let $I_{n}$ be the set consisting of maps $\iota : \{ 1, 2, \cdots, n\} \rightarrow \{ 0, 1 \}$.
We let $|\iota|$ denote the cardinality of the set $\{ i \mid \iota(i) = 1, \, 1 \leq i \leq n \}$.
For each generator $r \otimes (x_{1}, x_{2}, \cdots, x_{n})$ of $C^{R}_{n}(X; \mathbb{Z}[X])$, where 
$r , x_1, \dots , x_n \in X$, we define
\[
\begin{split}
 r(\iota) & = r * (x_{1}^{\iota(1)} x_{2}^{\iota(2)} \cdots x_{n}^{\iota(n)}) \in X, \\
 x(\iota,i) & = x_{i} * (x_{i+1}^{\iota(i+1)} x_{i+2}^{\iota(i+2)} \cdots x_{n}^{\iota(n)}) \in X,
\end{split}
\]
for any $\iota \in I_n$.
Fix an element $q \in X$.
For each $n \geq 1$, we define a homomorphism
\[
 \varphi : C^{R}_{n}(X; \mathbb{Z}[X]) \longrightarrow C^{\Delta}_{n+1}(X) \otimes_{\mathbb{Z}[G_X]} \mathbb{Z}
\]
by
\begin{equation}
\label{eq:main_map}
\varphi( r \otimes (x_{1}, x_{2}, \cdots, x_{n})) = \sum_{\iota \in I_{n}} (-1)^{|\iota|} (q, r(\iota), x(\iota,1), x(\iota,2), \dots, x(\iota,n)).
\end{equation}
For example, in the case $n = 2$ (see Figure \ref{fig:main_map}),
\[
\begin{split}
 \varphi(r \otimes (x, y)) = (q, r, x, y) - (q, r * x, x, y) - (q, r * y, x * y, y) + (q, (r*x)*y, x* y, y),
\end{split}
\]
\begin{figure}
\input{main_map.pstex_t}
\caption{} 
\label{fig:main_map}
\end{figure}
and in the case $n=3$,
\[
\begin{split}
 \varphi(r \otimes &(x, y, z)) = \\
 & (q, r, x, y, z) - (q, r * x, x, y, z) \\ 
 - \> & (q,r * y, x * y, y, z) - (q, r * z, x * z, y * z, z) \\
 + \> & (q, (r * x)*y, x * y, y, z) + (q, (r* x) * z, x * z, y * z, z) \\
 + \> & (q, (r * y) * z, (x * y) * z, y * z, z) - (q, ((r * x)*y)*z, (x * y) * z, y * z, z).
\end{split}
\]
\begin{theorem}[\cite{inoue-kabaya}]
The map $ \varphi : C^{R}_{n}(X; \mathbb{Z}[X]) \longrightarrow C^{\Delta}_{n+1}(X) \otimes_{\mathbb{Z}[G_X]} \mathbb{Z}$ is a chain map.
\end{theorem}
Therefore $\varphi$ induces a homomorphism $\varphi_* : H^R_n(X; \mathbb{Z}[X]) \to H^{\Delta}_{n+1}(X;\mathbb{Z}) $.
We remark that the induced map $\varphi_* : H^R_n(X; \mathbb{Z}[X]) \to H^{\Delta}_{n+1}(X;\mathbb{Z}) $ does not depend on the choice of $q \in X$ (\cite{inoue-kabaya}).

In general, it is easier to construct cocycles of $H^{\Delta}_{n+1}(X)$ from group cocycles of some group related to $X$ than those of $H^R_n(X; \mathbb{Z}[X])$.
If we have a function $f$ from $X^{k+1}$ to some abelian group $A$ satisfying 
\begin{enumerate}
\item $\displaystyle\sum_{i=0}^{k+1} (-1)^i f(x_0, \dots, \widehat{x_i}, \dots ,x_{k+1}) = 0$, 
\item $f(x_0 * y, \dots, x_k * y) =  f(x_0, \dots, x_k)$ for any $y \in X$, 
\item $f(x_0, \dots , x_{k} ) = 0$ if $x_i=x_{i+1}$ for some $i$, 
\end{enumerate}
then $f$ is a cocycle of $H_{\Delta}^{k}(X;A)$ and $\varphi^* f$ is a cocycle of $H_Q^{k-1}(X; \mathrm{Func}(X,A))$.
Moreover, $\varphi^* f$ can be regarded as a cocycle in $H_Q^k(X; A)$ by 
\[
(\varphi^* f) (r, x_1, \dots, x_{k-1})  = (\varphi^* f) (x_1, \dots, x_{k-1})(r).
\]
We will construct functions satisfying these three conditions from group cocycles.

\section{Cocycles of dihedral quandles}
\label{sec:dihedral}
For any integer $p>2$,
let $R_p$ denote the cyclic group $\mathbb{Z}/p$ with quandle operation defined by $x*y = 2y-x \mod p$. 
Actually this operation satisfies the quandle axioms. 
The quandle $R_p$ is called the \emph{dihedral quandle}.
In this section, we construct quandle cocycles of $R_p$ from group cocycles of $G=\mathbb{Z}/p$.
In the next section, we will propose a general construction of quandle cocycles from group cocycles.

\subsection{Group cohomology of cyclic groups}
Let $G$ be the cyclic group $\mathbb{Z}/p$ ($p$ is an integer greater than $2$).
The first cohomology $H^1(G; \mathbb{Z}/p) = \mathrm{Hom}(\mathbb{Z}/p, \mathbb{Z}/p)$ is generated by the 1-cocycle $b_1$ defined by
\[
b_1(x) = x.
\]
The connecting homomorphism $\delta : H^1(G, \mathbb{Z}/p) \to H^2(G; \mathbb{Z})$ of the long exact sequence corresponding to
$0 \to \mathbb{Z} \to \mathbb{Z} \to \mathbb{Z}/p \to 0$ maps $b_1$ to a generator of $H^2(G; \mathbb{Z})$, 
and the reduction $H^2(G; \mathbb{Z}) \to  H^2(G; \mathbb{Z}/p)$ maps it to a generator $b_2$ of $H^2(G; \mathbb{Z}/p)$.
Explicitly we have
\begin{equation}
\label{eq:2_cocycle_of_cyclic_group}
b_2(x,y) = \frac{1}{p} (\overline{y} - \overline{x+y} + \overline{x})
=
\left\{ 
\begin{array}{ll}
1 & \textrm{if $\overline{x}+\overline{y} \geq p$} \\  
0 & \textrm{otherwise}
\end{array} \right. 
\end{equation}
where $\overline{x}$ is an integer $0 \leq \overline{x} < p$ such that $\overline{x} \equiv x \mod p$.
Cup products of $b_1$'s and $b_2$'s are also cocycles.
Moreover, when $p$ is an odd prime, it is known that 
any element of $H^*(G; \mathbb{Z}/p)$ can be presented by a cup product of $b_1$'s and $b_2$'s, see e.g. \cite[Proposition 3.5.5]{benson}.
We remark that $b_1$ and $b_2$ and their products are normalized cocycles.

\subsection{Cocycle of $R_p$}
\label{subsec:3_cocycle_of_dihedral}
For an integer $p > 2$, let $f$ be a normalized $k$-cocycle of $H^k(G, \mathbb{Z}/p)$.
Regarding $R_p$ as $G = \mathbb{Z}/p$, 
we obtain a map $f: (R_p)^{k+1} \to \mathbb{Z}/p$ satisfying 
\begin{enumerate}
\item[(1)] $\displaystyle\sum_{i=0}^{k+1} (-1)^i f(x_0, \dots, \widehat{x_i}, \dots ,x_{k+1}) = 0$, 
\item[(3)] $f(x_0, \dots , x_{k} ) = 0$ if $x_i=x_{i+1}$ for some $i$,
\end{enumerate}
by using the homogeneous notation (\S \ref{subsec:group_homology}).
If $f$ also satisfies the condition 
\begin{enumerate}
\item[(2)] $f(x_0 * y, \dots, x_k * y) =  f(x_0, \dots, x_k)$ for any $y \in R_p$,
\end{enumerate}
then $f$ gives rise to a quandle $k$-cocycle of $H_Q^k(R_p; \mathbb{Z}/p)$ by the construction of Section \ref{sec:main_map}.  
Define $\tilde{f}: (R_p)^{k+1} \to \mathbb{Z}/p$ by
\begin{equation}
\label{eq:average_of_cocycle}
\tilde{f}(x_0, \dots, x_k) = f(x_0, \dots, x_k) + f(-x_0,\dots ,-x_k).
\end{equation}
Then $\tilde{f}$ satisfies the condition (2) by the left invariance of the homogeneous group cocycle.
It is easy to check that $\tilde{f}$ also satisfies the condition (1) and (3).
So we obtain a quandle $k$-cocycle.

We give an explicit presentation of the $3$-cocycle arising from $b_1 b_2 \in H^3(G; \mathbb{Z}/p)$.
Let 
\[
d(x,y)= b_2(x,y) - b_2(-x,-y)  
\]
then $d$ is a 2-cocycle. (We can check that $d$ is cohomologous to $2 b_2$.)
Then by the definition (equation (\ref{eq:average_of_cocycle})), $\widetilde{b_1b_2}$ is given by 
\[
[x|y|z] \mapsto x \cdot d(y,z).
\]
By definition we have
\begin{equation}
\label{eq:symmetric}
d(-x,-y) = -d(x,y)
\end{equation}
and 
\begin{equation}
\label{eq:formula_of_d}
d(x,y) =  \left\{ 
\begin{array}{ll}
1  & \textrm{if $\overline{x} + \overline{y} > p$} \\  
-1 & \textrm{if $\overline{x} + \overline{y} < p$, $x \neq 0$ and $y \neq 0$} \\
0  & \textrm{otherwise}. \\
\end{array} \right. 
\end{equation}
We remark that the cocycle $d$ can be easily understood geometrically.
Identify $i \in \mathbb{Z}/p$ with the complex number $\zeta^i$ where $\zeta = \exp (2 \pi \sqrt{-1}/ p)$.  
Then $d(x,y)=-1$ if $(0,x,x+y)$ is counterclockwise, $d(x,y) = +1$ if $(0,x,x+y)$ is clockwise 
and $d(x,y)=0$ if $(0,x,x+y)$ is degenerate (Figure \ref{fig:2-cocycle}). 
This interpretation and the equation (\ref{eq:symmetric}) make various calculations easy.
\begin{figure}
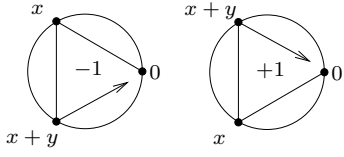

\input{2-cocycle_neg.pstex_t}
\input{2-cocycle_pos.pstex_t}
\caption{The value of $d(x,y)$.}
\label{fig:2-cocycle}
\end{figure}

\begin{proposition}
The quandle $3$-cocycle arising from $b_1 b_2 \in H^3(G; \mathbb{Z}/p)$ has the following presentation:
\begin{equation}
\label{eq:the_cocycle}
(x,y,z) \mapsto 2z (d(y-x,z-y) + d(y-x,y-z) )  \quad (x,y,z \in R_p).
\end{equation}
This is a non-trivial quandle 3-cocycle of $R_p$ with $\mathbb{Z}/p$ coefficient.
\end{proposition}
\begin{proof}
In (\ref{eq:main_map}), since the map $\varphi_*$ does not depend on the choice of $q \in R_p$, 
we let $q=0$. Then we have
\[
\begin{split}
\varphi(x&,y,z) = (0,x,y,z) - (0,x*y,y,z) - (0,x*z, y*z, z) + (0,(x*y)*z, y*z, z) \\
=& (0,x,y,z) - (0,2y-x, y,z) - (0,2z-x,2z-y,z) + (0,2z-2y+x,2z-y,z) \\
\end{split}
\] 
for $x,y,z \in R_p$.
Rewriting in inhomogeneous notation, this is equal to
\[
\begin{split}
& [x| y-x|z-y] - [2y-x|x-y|z-y] \\ 
&- [2z-x| x-y|y-z] + [2z-2y+x| y-x|y-z]. \\
\end{split}
\]
The evaluation of $\widetilde{b_1b_2}$ on this cycle is 
\[
\begin{split}
& x \cdot d(y-x,z-y) - (2y-x) d(x-y,z-y) \\ 
&- (2z-x) d(x-y,y-z) + (2z-2y+x) d(y-x,y-z) \\
=& x \cdot d(y-x,z-y) + (2y-x) d(y-x,y-z)  \\
&+ (2z-x) d(y-x,z-y) + (2z-2y+x) d(y-x,y-z) \\
=& 2z \cdot d(y-x,z-y) + 2z \cdot d(y-x,y-z). \\
\end{split}
\]

We will see that this cocycle is non-trivial because the evaluation on the cycle given by a shadow coloring $\mathcal{S}$ 
of the $(2,p)$-torus link (Figure \ref{fig:(2,p)-torus}) is non-zero.
Color two arcs by $x, y \in R_p$ as in the Figure \ref{fig:(2,p)-torus}.
Then other arcs must be colored by $(i+1) y - ix$ by the relations at the crossings.
Let $r$ be the color of the central region.
Then we have
\[
C(\mathcal{S}) = \sum_{i=0}^{p-1} r \otimes (iy -(i-1)x, (i+1)y - ix ) .
\]
We assume that $r=0$ and $x=0$.
Evaluation of the cocycle on $C(\mathcal{S})$ is equal to
\[
\begin{split}
&  \sum_{i=0}^{p-1} 2 (i+1) y  ( d(iy, y) + d(iy, -y)) \\ 
&= 2y \sum_{i=0}^{p-1} (i+1)( d(iy, y) - d(-iy,y) ) \quad  \textrm{( by (\ref{eq:symmetric}) )} \\
&= 2y \sum_{i=0}^{p-1} ( (i+1)d(iy, y) + (i-1)d(iy,y) ) \\
&= 4y \sum_{i=0}^{p-1} i \cdot d(iy, y).  \\
\end{split}
\]
By the next lemma, this is equal to $- 4 y^2  \mod p$.
\begin{figure}
\input{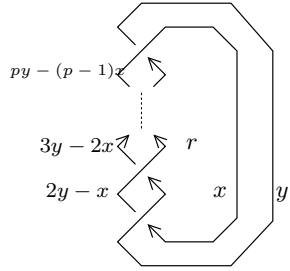}
\caption{A shadow coloring of $(2,p)$-torus knot by $R_p$. (For any $x,y, r \in R_p$.)}
\label{fig:(2,p)-torus}
\end{figure}
\end{proof}

\begin{lemma}
Let $p > 2$ be an integer. For $0 < y < p$, we have 
\[
\sum_{i=0}^{p-1} i \cdot d(iy, y) = 
\left\{ 
\begin{array}{ll}
-y & \textrm{$p$: odd}  \\
p/2-y & \textrm{$p$: even} \\
\end{array} \right. 
\]
When $y=0$, the left hand side is $0$.
\end{lemma}
\begin{proof}
When $y=0$, this is straightforward since $d(x,0) = 0$ for any $0 \leq x < p$.
When $y$ is a unit in $\mathbb{Z}/p$, by (\ref{eq:formula_of_d}) we have
\[
\begin{split}
\sum_{i=0}^{p-1} i \cdot d(iy, y) & = \sum_{i=0}^{p-1} \frac{i}{y} \cdot d(i, y) \\
& = \frac{1}{y} ( -1 -2 - \dots - (p-y-1) \\ 
   & \quad \quad + (p-y+1) + (p-y+2) + \dots  + (p-1) ) \\
& = \frac{1}{y} \left( \frac{(p-y-1)(y-p)}{2} + \frac{(y-1)(2p-y)}{2} \right) \\
& = \frac{1}{y} \left( - y^2 -\frac{p}{2} + 2py -  \frac{p^2}{2} \right) \\
& \equiv -y - \frac{p}{2} \cdot \frac{1}{y} \\
\end{split}
\]
This is equal to $-y$ when $p$ is odd.
When $p$ is even, $- p/2 \cdot 1/y  \equiv p/2 \mod p$ since $1/y$ is a unit in $\mathbb{Z}/p$.

When $y$ is not a unit in $\mathbb{Z}/p$, let $c$ be the greatest common divisor of $y$ and $p$.
Then 
\[
\begin{split}
\sum_{i=0}^{p-1} i \cdot d(iy, y) & = c  \sum_{j=0}^{p/c-1} \frac{j}{y/c} \cdot d(jc, y) \\
& = \frac{c}{y/c} \sum_{j=0}^{p/c-1} j \cdot d(j, y/c). \\
\end{split}
\]
Since $y/c$ is a unit in $\mathbb{Z}/(p/c)$, this reduces to the previous case.
\end{proof}

Since $2$ is divisible in $\mathbb{Z}/p$ when $p$ is odd, we have
\begin{corollary}
When $p > 2$ is an odd number, 
\begin{equation}
(x,y,z) \mapsto z (d(y-x,z-y) + d(y-x,y-z) )
\end{equation}
is a non-trivial quandle 3-cocycle of $R_p$ with $\mathbb{Z}/p$ coefficient.
\end{corollary}
When $p$ is prime, 
it is known that $\dim_{\mathbb{F}_p} H_Q^3(R_p; \mathbb{F}_p) = 1$.
Therefore our cocycle is a constant multiple of the Mochizuki 3-cocycle \cite{mochizuki}.
We remark that when $p$ is prime, $\dim_{\mathbb{F}_p} H_Q^n(R_p; \mathbb{F}_p)$ was calculated by Nosaka \cite{nosaka_alexander} for any $n$,
moreover he gave a system of generators of $H_Q^n(R_p; \mathbb{F}_p)$.

When $p$ is an odd integer, Nosaka showed in \cite{nosaka_surface} that $H^Q_3(R_p; \mathbb{Z}) \cong \mathbb{Z}/p$.
Since $H^Q_2(R_p; \mathbb{Z})$ is zero, we have $H_Q^3(R_p; \mathbb{Z}/p) \cong \mathbb{Z}/p$. 
This means that there exists a non-trivial quandle 3-cocycle of $R_p$ with $\mathbb{Z}/p$ coefficient.

\section{General construction}
\label{sec:gen_const}
In this section, we generalize the construction in the previous section to wider classes of quandles. 
We will construct a quandle cocycle of a faithful homogeneous quandle $X$ from a group cocycle of $\mathrm{Aut}(X)$
when an obstruction living in the second cohomology of $\mathrm{Aut}(X)$ vanishes.

\subsection{}
Let $G$ be a group. Fix an element $h \in G$.
Let $\Conj(h) = \{ g^{-1} h g | g \in G\}$.
Now $\Conj(h)$ has a quandle operation by $x*y = y^{-1} x y$.
In this section, we shall construct a quandle cocycle of $\Conj(h)$ from group cocycle of $G$.
First we shall show that this class of quandles is not so special.

Let $X$ be a quandle.
We denote the group of the quandle automorphisms of $X$ by $\mathrm{Aut}(X)$.
We regard an automorphism acts on $X$ from the right.
For $x \in X$, let $S(x)$ be the map which sends $y$ to $y*x$.
By the axioms (Q2) and (Q3), $S(x)$ is a quandle automorphism.
$X$ is called \emph{faithful} if $S: X \to \mathrm{Aut}(X)$ is injective.
A quandle $X$ is \emph{homogeneous} if $\mathrm{Aut}(X)$ acts on $X$ transitively.
The following lemma was essentially shown in Theorem 7.1 of \cite{joyce}, but we include a proof for completeness.
\begin{lemma}
Every faithful homogeneous quandle $X$ is represented by $\Conj(h)$ with some group $G$ and $h \in G$.
\end{lemma}
\begin{proof}
For $x \in X$ and $g \in \mathrm{Aut}(X)$, we have $S(xg) = g^{-1} S(x) g$.
In fact, $(y)S(xg) = y*(xg) = (yg^{-1}*x)g = (y) g^{-1}S(x)g$ for any $y \in X$.

Let $G=\mathrm{Aut}(X)$ and fix an element $x_0 \in X$.
Put $h=S(x_0)$.
Because $X$ is homogeneous, for any $x \in X$, there exists $g \in G$ such that $x = x_0 g$.  
So we have $S(x) = S(x_0 g) =g^{-1} h g$, namely $S(x) \in \Conj(h)$.
Therefore we obtained a homomorphism $S: X \to \Conj(h)$. This is surjective 
since $g^{-1} S(x_0) g = S(x_0 g)$, and injective since $X$ is faithful.
\end{proof}

Let $Z(h) = \{ g \in G | gh = hg \}$ be the centralizer of $h$ in $G$.
\begin{lemma}
A map 
\[
\begin{matrix}
\Conj(h) & \to & Z(h) \backslash G \\ 
\textrm{\rotatebox{90}{$\in$}} & & \textrm{\rotatebox{90}{$\in$}} \\
g^{-1} h g & \mapsto & Z(h)g
\end{matrix}
\] 
is well-defined and bijective.
\end{lemma}
\begin{proof}
Let $g_1^{-1} h g_1 = g_2^{-1} h g_2$.
Then $(g_1 g_2^{-1})^{-1} h (g_1 g_2^{-1}) = h$, so $g_1 g_2^{-1} \in Z(h)$ and $g_1 \in Z(h) g_2$.
This means that $g_1$ and $g_2$ belong to the same right coset.
Therefore the map is well-defined. By a similar argument, we can show the injectivity.
Surjectivity is trivial by definition.
\end{proof}

From now on we study the quandle structure on $Z(h) \backslash G$ and construct a section of
the projection $\pi: G \to Z(h) \backslash G$.
The quandle operation on $\Conj(h)$ induces a quandle operation on $Z(h) \backslash G$:
\[
\begin{split}
Z(h)g_1 * Z(h)g_2 & \leftrightarrow (g_1^{-1} h g_1) * (g_2^{-1} h g_2) \\
&= (g_2^{-1} h g_2)^{-1}  (g_1^{-1} h g_1)  (g_2^{-1} h g_2)  \\
&= (g_1 g_2^{-1} h g_2)^{-1} h (g_1 g_2^{-1} h g_2)  \\
& \leftrightarrow Z(h) g_1 (g_2^{-1} h g_2). 
\end{split}
\]
This quandle operation on $Z(h) \backslash G$ lifts to a quandle operation $\widetilde{*}$ on $G$ by
\begin{equation}
\label{eq:lift_of_the_operation}
g_1 \widetilde{*} g_2 = h^{-1} g_1 (g_2^{-1} h g_2) \quad (g_1, g_2 \in G).
\end{equation}
We can easily check that $\widetilde{*}$ satisfies the quandle axioms (the inverse operation is given by $g_1 \widetilde{*}^{-1} g_2 = h g_1 g_2^{-1}h^{-1}g_2$)
and the projection map $\pi : G \to Z(h) \backslash G$ is a quandle homomorphism.
We remark that the quandle operation given by (\ref{eq:lift_of_the_operation}) has been already studied by Joyce \cite{joyce} and Eisermann \cite{eisermann}.

Let $s: Z(h) \backslash G \to G$ be a section of $\pi$, i.e. a map (not a homomorphism) satisfying $\pi \circ s = \mathrm{id}$.
Since $s(x*y)$ and $s(x) \widetilde{*} s(y)$ are in the same coset in $Z(h) \backslash G$,
there exists an element $c(x, y) \in Z(h)$ satisfying 
\[
s(x) \widetilde{*} s(y) = c(x,y) s(x*y) .
\]
\begin{lemma}
\label{lem:section}
If $Z(h)$ is an abelian group, 
$c : Z(h) \backslash G \times Z(h) \backslash G  \to Z(h)$ is a quandle 2-cocycle.
If the cocycle $c$ is cohomologous to zero, we can change the section $s$ to satisfy $s(x*y) = s(x) \widetilde{*} s(y)$. 
\end{lemma}
\begin{proof}
For $c_1,c_2 \in Z(h)$ and $g_1, g_2 \in G$, we have
\[
(c_1 g_1) \widetilde{*} (c_2 g_2) = h^{-1} c_1 g_1 g_2^{-1} c_2^{-1} h c_2 g_2 = c_1 h^{-1} g_1 g_2^{-1} h g_2 = c_1 (g_1 \widetilde{*} g_2).
\]
Therefore 
\begin{equation}
\label{eq:ditribution_1}
\begin{split}
(s(x) \widetilde{*} s(y)) \widetilde{*} s(z) &= (c(x,y) s(x*y) ) \widetilde{*} s(z) \\
&= c(x,y) (s(x*y) \widetilde{*} s(z)) \\
&= c(x,y) c(x*y,z) s((x*y)*z) 
\end{split}
\end{equation}
and
\begin{equation}
\label{eq:ditribution_2}
\begin{split}
(s(x) \widetilde{*} s(z)) \widetilde{*} (s(y) \widetilde{*} s(z)) &= (c(x,z) s(x*z)) \widetilde{*} (c(y,z) s(y*z)) \\
&= c(x,z) c(x*z,y*z) s( (x*z)*(y*z) ) \\
\end{split}
\end{equation}
for any $x,y,z \in Z(h) \backslash G$.
Comparing (\ref{eq:ditribution_1}) and (\ref{eq:ditribution_2}), we have
\[
c(x,z) c(x*y,z)^{-1} c(x,y)^{-1} c(x*z,y*z) = 1.
\]
By $s(x) \widetilde{*} s(x) = s(x)$, we also have $c(x,x) =1$.

If $c$ is cohomologous to zero, then there exists a map $b : Z(h) \backslash G \to Z(h)$ satisfying $c(x,y) = b(x) b(x*y)^{-1}$.
Put $s'(x) = b(x)^{-1} s(x)$, then $s'$ satisfies $s'(x) \widetilde{*} s'(y) = s'(x*y)$.
\end{proof}

\begin{remark}
The $2$-cocycle $c$ has already appeared in \cite{eisermann} in a similar context.
\end{remark}

\begin{example}
Let $G$ be the dihedral group $D_{2p} = \langle h,x | h^2 = x^p = h x h x =1 \rangle$ where $p$ is an odd number greater than $2$. 
Then we have $Z(h) = \{ 1, h\}$ and 
$\Conj(h) = \{ x^{-i} h x^{i} | i = 0, 1, \dots, p-1\} = \{ hx^{2i} | i = 0, \dots,p-1 \}$.
We can identify $x^{-i} h x^i \in \Conj(h)$ with $i \in R_p = \{ 0,1,2, \dots, p-1\}$.
Define a section $s: Z(h) \backslash G \to G$ by 
\[
\begin{matrix}
\Conj(h) & \cong & Z(h) \backslash G & \xrightarrow{s}{} & G\\
\textrm{\rotatebox{90}{$\in$}} & & \textrm{\rotatebox{90}{$\in$}}  &  & \textrm{\rotatebox{90}{$\in$}} \\
x^{-i}h x^i & \leftrightarrow & Z(h)x^i & \mapsto & hx^i\\
\end{matrix}
\]
Then we have
\[
\begin{split}
s( & Z(h)x^{i} * Z(h)x^{j}) = s( Z(h) x^{2j-i}) = h x^{2j-i} \\
& = h^{-1} (h x^{i}) (x^{-j} h x^{j} ) = s(Z(h)x^i) \widetilde{*} s(Z(h)x^j).
\end{split}
\]
Therefore $c(x,y) = 0$ for any $x,y \in R_p$.
\end{example}

\subsection{}
Let $G$ be a group.  Fix $h \in G$ with $h^{l}=1$ ($l > 1$). 
In the following we assume:
\begin{assumption}
\label{asm:vanishing}
$Z(h)$ is abelian and the 2-cocycle corresponding to $G \to Z(h) \backslash G$ is cohomologous to zero. 
\end{assumption}
Under the assumption, we can take a section $s: Z(h) \backslash G \to G$ satisfying $s(x*y) = s(x) \widetilde{*} s(y)$ by Lemma \ref{lem:section}. 
Let $f: G^{k+1} \to A$ be a normalized group $k$-cocycle of $G$ in the homogeneous notation, where $A$ is an abelian group.
Then $f$ satisfies 
\begin{enumerate}
\item $\displaystyle\sum_{i=0}^{k+1} (-1)^i f(x_0, \dots, \widehat{x_i}, \dots ,x_{k+1}) = 0$, 
\item $f(gx_0, \dots, gx_k) =  f(x_0, \dots, x_k)$ for any $g \in G$ (left invariance), 
\item $f(x_0, \dots , x_{k} ) = 0$ if $x_i=x_{i+1}$ for some $i$.
\end{enumerate}
Since it is convenient to use a right invariant function in the following construction, we modify the condition (2) by  
\begin{enumerate}
\item[(2')] $f(x_0 g, \dots, x_k g) =  f(x_0, \dots, x_k)$ for any $g \in G$ (right invariance), 
\end{enumerate}
by replacing $f(x_0,\dots,x_k)$ by $ f(x_0^{-1}, \dots, x_k^{-1})$.
Define  $\tilde{f} : \Conj(h)^{k+1} \to A$ by
\[
\tilde{f}(x_0. \dots, x_k) = \sum_{i=0}^{l-1} f(h^i s(x_0), \dots, h^i s(x_k)) 
\]
for $x_0, \dots , x_k \in \Conj(h)$.

\begin{proposition}
The function $\tilde{f}$ satisfies the conditions (1), (2) and (3) of Section \ref{sec:main_map}. 
Therefore $\tilde{f}$ gives rise to a $k$-cocycle of $H_{\Delta}^k( \Conj(h) ;A)$. 
\end{proposition}
\begin{proof}
It is clear that (1) and (3) are satisfied from the conditions on a normalized group cocycle in the homogeneous notation.
We only have to check the second property.
\[
\begin{split}
\tilde{f}&(x_0*y, \dots, x_k*y) \\
&= \sum_{i=0}^{l-1} f(h^i s(x_0*y), \dots, h^i s(x_k*y)) \\
&= \sum_{i=0}^{l-1} f(h^i s(x_0) \widetilde{*} s(y), \dots, h^i s(x_k) \widetilde{*} s(y)) \\
&= \sum_{i=0}^{l-1} f(h^{i-1} s(x_0) (s(y)^{-1} h s(y)), \dots, h^{i-1} s(x_k) (s(y)^{-1} h s(y))) \\
&= \sum_{i=0}^{l-1} f(h^{i-1} s(x_0), \dots, h^{i-1} s(x_k)) \quad \textrm{(right invariance)}\\
&= \tilde{f}(x_0, \dots, x_k).
\end{split}
\]
\end{proof}

Combining with the arguments of Section \ref{sec:main_map}, we have
\begin{corollary}
If $Z(h)$ is abelian and the 2-cocycle corresponding to $G \to Z(h) \backslash G$ is cohomologous to zero, 
then there is a homomorphism
\[
H^n(G; A) \to H_Q^n(\Conj(h); A)
\]
for any abelian group $A$.
\end{corollary}

\subsection{}
We return to the case of $R_p$ discussed in the previous section.
We assume that $p$ is an odd integer greater than $2$.
Let $G$ be the dihedral group $D_{2p} = \langle h,x | h^2 = x^p = h x h x =1 \rangle$.
Consider the short exact sequence 
\begin{equation}
\label{eq:short_exact}
0 \to \mathbb{Z}/p \to D_{2p} \to \mathbb{Z}/2 \to 0.
\end{equation}
We regard $\mathbb{Z}/2$ as $\{1,h\}$ by taking coset representatives in $D_{2p}$.
Then $\mathbb{Z}/2$ acts on $\mathbb{Z}/p$ by $h (x^i) = x^{-i}$.
This induces the restriction map 
\[
H^*(D_{2p}; \mathbb{Z}/p) \to H^*(\mathbb{Z}/p; \mathbb{Z}/p)^{\mathbb{Z}/2}.
\]
We can show that this homomorphism is an isomorphism (Proposition III.10.4 of \cite{brown}).
To obtain a group cocycle of $D_{2p}$ from a group cocycle of $\mathbb{Z}/p$, we need the inverse map, which is called the \emph{transfer}.
The transfer map is described as follows (see also \cite{brown}).
Let $r$ be the map $D_{2p} \to \mathbb{Z}/p$ defined by $r(x^i)=x^i$ and $r(hx^i)=x^{-i}$.
For a cocycle $f$ of $H^n(\mathbb{Z}/p; \mathbb{Z}/p)^{\mathbb{Z}/2}$, the image $f'$ of the transfer is given by 
\[
f'(x_0,\dots, x_n) = f(r x_0,\dots, r x_n)+f(h r x_0, \dots, h r x_n)
\]
in the homogeneous notation.
When restricted to the image of $s: Z(h) \backslash G \to G$, this is equal to the map defined in (\ref{eq:average_of_cocycle}).
Applying our construction for this group cocycle, we obtain a quandle 3-cocycle of $R_p$, which is twice the cocycle constructed in the previous section.

\subsection{}
\label{subsec:quandle_to_group}
We end this section by giving another homomorphism from $H^n(G; A)$ to $H_Q^n(\Conj(h); A)$ arising from more general context.

Let $X$ be a quandle and $M$ be a right $\mathbb{Z}[G_X]$-module.
We can construct a map from the rack homology $H^R_n(X;M)$ to the group homology $H_n(G_X; M)$.
The following lemma is well-known, e.g. see Lemma 7.4 of Chapter I of \cite{brown}.
\begin{lemma}
Let $ \dots \to P_1 \to P_0 \to M \to 0 $ be a chain complex where $P_i$ are projective (e.g. free).
Let $ \dots \to C_1 \to C_0 \to N \to 0 $ be an acyclic complex.
Any homomorphism $M \to N$ can be extended to a chain map from $\{ P_*\}$ to $\{C_*\}$.
Moreover such a chain map is unique up to chain homotopy.
\end{lemma}
So there exists a unique chain map from $C^R_*(X)$ to $C_*(G_X)$ up to homotopy.
This map induces $M \otimes_{\mathbb{Z}[G_X]} C^R_*(X) \to M \otimes_{\mathbb{Z}[G_X]} C_*(G_X)$ and then
$H^R_n(X;M) \to H_n(G_X; M)$.
Using normalized chains in group homology, we can also construct a map $H^Q_n(X;M) \to H_n(G_X; M)$.
We give an explicit chain map. 
Let $(x_1, \dots , x_n)$ be a generator of $C^R_n(X)$.
We define a map $\psi$ by
\[
\psi((x_1, \dots , x_n))= \sum_{\sigma \in \mathfrak{S}_n} \mathrm{sgn}(\sigma) [y_{\sigma,1}|\cdots|y_{\sigma,i}|\cdots|y_{\sigma,n}]
\]
where $y_{\sigma,i} \in X$ is defined for a permutation $\sigma$ and $i \in \{1, \dots, n\}$ as follows.
Let $j_1, \dots, j_{i}<i$ be the maximal set of numbers satisfying $\sigma(i) < \sigma(j_1) <  \sigma(j_2) < \dots < \sigma(j_{i}) $.
Then define
\[
y_{\sigma,i}= x_{\sigma(i)}*(x_{\sigma(j_{1})} x_{\sigma(j_{2})} \cdots x_{\sigma(j_{i})}).
\]
A graphical picture of this map is given in Figure \ref{fig:quandle_cycle_to_group_cycle}.
For example, when $n=3$,
\[
\begin{split}
\psi((x,y,z)) = &[x|y|z] - [x|z|y*z]+[y|z|(x*y)*z]-[y|x*y|z] \\
&+[z|x*z|y*z]-[z|y*z|(x*y)*z],
\end{split}
\]
for $(x,y,z) \in C^R_3(X)$.
Dually, we also have a map $H^n(G_X;M) \to H_Q^n(X;M)$.

\begin{figure}
\input{quandle_group.pstex_t}
\caption{} 
\label{fig:quandle_cycle_to_group_cycle}
\end{figure}

We apply this map for $\Conj(h)$. 
Since there exists a natural homomorphism from the associated group $G_{\Conj(h)}$ to $G$, we have a homomorphism
\[
H^n(G;A) \to H^n(G_{\Conj(h)};A) \to H_Q^n(\Conj(h);A).
\]

Fenn, Rourke and Sanderson defined the \emph{rack space} $BX$ in \cite{FRS}. 
Since $\pi_1(BX)$ is isomorphic to $G_X$, there exists a unique map, up to homotopy, from $BX$ to the Eilenberg-MacLane space $K(G_X,1)$ 
which induces the isomorphism between their fundamental groups.
This map induces a homomorphism $H^n(G_X;M) \to H_Q^n(X;M)$, which is equal to the map we have constructed in this subsection.
Clauwens showed in Proposition 25 of \cite{clauwens} that this map vanishes under some conditions on $X$ and $M$. 
In particular, when $p$ is odd prime, $H^n(D_{2p}; \mathbb{Z}/p) \to H_Q^n(R_p; \mathbb{Z}/p)$ vanishes for $n > 0$.

\section{Quandle cycle and branched cover}
\label{sec:quandle_cyc_as_cyc_bra_cov}
In this section, we study the dual of the previous construction.
We will show that the cycle $C(\mathcal{S})$ associated with a shadow coloring $\mathcal{S}=(\mathcal{A}, \mathcal{R})$ of a knot $K$ 
gives rise to a group cycle represented by a cyclic branched covering along $K$ and the representation induced from the arc coloring $\mathcal{A}$.

\subsection{}
Let $X$ be a quandle.
Let $D$ be a diagram of a knot $K$.
For a shadow coloring $\mathcal{S} = (\mathcal{A}, \mathcal{R})$ of $D$ whose arcs and regions are colored by $X$, define $\mathcal{A}*a$ and $\mathcal{R}*a$ 
for $a \in X$ by
\[
(\mathcal{A}*a )(x) = \mathcal{A}(x)*a, \quad (\mathcal{R}*a )(r) = \mathcal{R}(r)*a \quad (\textrm{for any arc $x$ and region $r$}). 
\] 
By the axiom (Q3), $\mathcal{S}*a = (\mathcal{A}*a, \mathcal{R}*a)$ is also a shadow coloring.

In the following, we assume that $X = \Conj(h)$ for some group $G$ and $h \in G$ and satisfies Assumption \ref{asm:vanishing}.
As in the previous section, let $s : \Conj(h) \cong Z(h) \backslash G  \to G$ be a section satisfying $s(a*b) = s(a) \widetilde{*} s(b)$ for $a,b \in X$. 
Let 
\[
\iota : C^{\Delta}_n(X) \to C_n(G; \mathbb{Z}): (a_0, \dots, a_n) \mapsto (s(a_0), \dots, s(a_n)).
\]
Composing with $\varphi : C^{Q}_{n}(X; \mathbb{Z}[X]) \longrightarrow C^{\Delta}_{n+1}(X)$, we have a map
\[
C^{Q}_{2}(X; \mathbb{Z}[X]) \xrightarrow{\varphi}{} C^{\Delta}_{3}(X) \xrightarrow{\iota}{} C_{3}(G; \mathbb{Z}). 
\]
For a shadow coloring $\mathcal{S}$, $ \iota \varphi (C(\mathcal{S}))$ is not a cycle of $C_3(G;\mathbb{Z})$ in general. 
But we can show 
\begin{theorem}
\label{thm:main_branched}
Let $\mathcal{S} = (\mathcal{A},\mathcal{R})$ be a shadow coloring of a diagram $D$ of a knot $K$ by $\Conj(h)$. 
Let $a \in \Conj(h)$ be the color of an arc of $D$.  
If $h^l = 1$, 
\begin{equation}
\label{eq:cycle_represented_by_cyclic_branched_cover}
\iota \varphi(C(\mathcal{S})) + \iota \varphi(C(\mathcal{S}*a)) +\iota \varphi(C(\mathcal{S}*a^2)) 
+ \dots + \iota \varphi(C(\mathcal{S}*a^{l-1})) \in C_3(G;\mathbb{Z})
\end{equation}
is a group cycle represented by the $l$-fold cyclic branched covering $\widehat{C}_l$ along the knot $K$ and 
the representation $\pi_1(\widehat{C}_l) \to G$ induced from the arc coloring $\mathcal{A}$.
\end{theorem}
\begin{proof}
Let $x_1, \dots, x_n$ be the arcs of the diagram $D$ such that $a = \mathcal{A}(x_1)$.
We denote the color $\mathcal{A}(x_i)$ by $a_i$ ($a_1 = a$).
The arc coloring $\mathcal{A}$ induces a representation $\rho : \pi_1(S^3 - K ) \to G$. 
Using $s : Z(h) \backslash G \to G$,  $\rho$ is given by
\[
\rho(x_i) = s(a_i)^{-1} h s(a_i) \in \Conj(h) \subset G.
\]
We have 
\[
\begin{split}
s(b* a_i) &= h^{-1} s(b) s(a_i)^{-1} h s(a_i) = h^{-1} s(b) \rho(x_i) \in G, \\
s(b *^{-1} a_i) &= h s(b) s(a_i)^{-1} h^{-1} s(a_i) = h s(b) \rho(x_i)^{-1} \in G,
\end{split}
\]
for any $b \in \Conj(h)$. So we have
\[
\begin{split}
s( (b * a_i ) *  a_1^{-s} ) &=  h^{s-1} s(b) \rho(x_i) \rho(x_1)^{-s}, \\
s(b * a_1^{-s+1}) & =  h^{s-1} s(b) \rho(x_1)^{-s+1}. \\
\end{split}
\]
Since 
\[
\rho(x_{i,s}) = \rho(x_1^{s-1} x_i x_1^{-s}) = \rho(x_1)^{s-1} \rho(x_i) \rho(x_1)^{-s}
\]
by (\ref{eq:covering_representation}), we have 
\begin{equation}
\label{eq:gluing_relation}
s(b * a_1^{-s+1}) = s( (b * a_i ) *  a_1^{-s} ) \rho(x_{i,s})^{-1}.
\end{equation}

\begin{figure}
\input{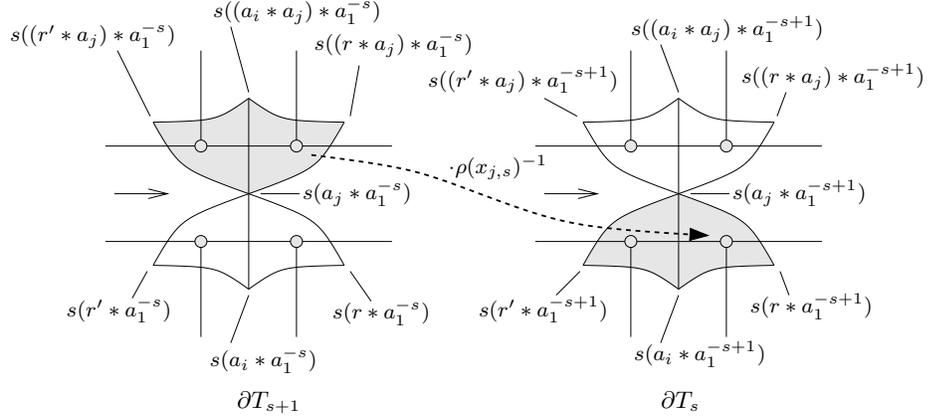}
\caption{A labeling at a crossing. ($r' = r*^{\pm 1} a_i$.)}
\label{fig:labeling_of_crossing}
\end{figure}

Let $T_s (s= 0,1, \dots l-1)$ be copies of the triangulation of a $3$-ball with a vertex ordering on each tetrahedra 
constructed in Section \ref{sec:cyc_rep_by_cyc_bra_cov} (Figure \ref{fig:cut_heegaard} and \ref{fig:ordering}).
Then we define a labeling of vertices of $T_s$ by $\iota \varphi (C(\mathcal{S}*a_1^{-s+1}))$ (see Figure \ref{fig:labeling_of_crossing}).
The vertices of the face ${F'}^{-}_{j,s+1} \subset T_{s+1}$ is labeled by 
\[
(s((r*a_j) *a_1^{-s}), \quad s( (a_i*a_j)*a_1^{-s}), \quad s(a_j*a_1^{-s}) )
\]
and the vertices of the face ${F'}^{+}_{j,s} \subset T_{s}$ is labeled by 
\[
(s( r*a_1^{-s+1}), \quad s(a_i*a_1^{-s+1}), \quad s(a_j*a_1^{-s+1}) ),
\]
where $r$ is the color of the region as indicated in Figure \ref{fig:labeling_of_crossing}.
By (\ref{eq:gluing_relation}) the labelings of the face ${F'}^{-}_{j,s+1}$ and ${F'}^{+}_{j,s}$ are related by  
\[
\begin{split}
(s( r*a_1^{-s+1}), & \quad  s(a_i*a_1^{-s+1}), \quad s(a_j*a_1^{-s+1}) )   \\
&= (s((r*a_j) *a_1^{-s}), \quad s( (a_i*a_j)*a_1^{-s}), \quad s(a_j*a_1^{-s}) ) \rho(x_{j,s})^{-1}.
\end{split}
\]
Therefore this gives a $G$-valued $1$-cocycle on $\widehat{T}$ as constructed in Section \ref{sec:cyc_rep_by_cyc_bra_cov}.
Thus the chain given by (\ref{eq:cycle_represented_by_cyclic_branched_cover}) is a cycle represented by the cyclic branched cover $\widehat{C}_l$ and 
the representation induced from the arc coloring $\mathcal{A}$.
\end{proof}

\subsection{}
We end this section by comparing the shadow cocycle invariant of the $(2,p)$-torus knot ($p$: odd) for the $3$-cocycle obtained in Section \ref{sec:dihedral} 
and the Dijkgraaf-Witten invariant of the lens space $L(p,1)$.

Let $G$ be a finite group and $\alpha$ be a cocycle of $H^3(G;A)$ where $A$ is an abelian group.
Dijkgraaf and Witten defined an invariant of closed oriented 3-manifolds for each $\alpha$.
For an oriented closed manifold $M$, it is defined by
\[
\sum_{\rho: \pi_1(M) \to G} \langle \rho^* \alpha,  [M] \rangle \in \mathbb{Z}[A]
\]
where $\rho^* \alpha$ is the pull-back of $\alpha$ by the classifying map  $M \to BG$ corresponding to $\rho$.
Since $L(p,1)$ is a double branched cover along the $(2,p)$-torus knot, 
the cycle obtained from a shadow coloring by $R_p$ gives rise to a group $3$-cycle of $D_{2p}$ represented by $L(p,1)$ by Theorem \ref{thm:main_branched}. 
Since every representation of $\pi_1(L(p,1)) \cong \mathbb{Z}/p$ into $D_{2p}$ reduces to a representation into $\mathbb{Z}/p$, 
it is natural to ask whether the shadow cocycle invariant for our quandle $3$-cocycle coincides with the Dijkgraaf-Witten invariant 
for $b_1b_2 \in H^3(\mathbb{Z}/p;\mathbb{Z}/p)$.
As we remarked in the introduction, these invariants coincide up to some constant by the result of Hatakenaka and Nosaka \cite{hatakenaka-nosaka} when $p$ is prime.


Let $\mathcal{S}$ be the shadow coloring indicated in Figure \ref{fig:(2,p)-torus}.
Since the homology class $[C(\mathcal{S})]$ does not change under the action of $R_p$ on the shadow coloring $\mathcal{S}$ (see Lemma 4.5 of \cite{inoue-kabaya}), 
we assume that $x=0$.
As computed in \S \ref{subsec:3_cocycle_of_dihedral}, the evaluation of the $C(\mathcal{S})$ at the $3$-cocycle derived from $b_1b_2$ 
is equal to $- 4 y^2  \mod p$.
Therefore the shadow cocycle invariant is 
\[
p \sum_{y=0}^{p-1} t^{-4y^2} = p \sum_{y=0}^{p-1} t^{-y^2}  \in \mathbb{Z}[t]/ (t^p-1) \cong \mathbb{Z}[\mathbb{Z}/p].
\]

We compute the Dijkgraaf-Witten invariant of the lens space $L(p,q)$ for the group $3$-cocycle $b_1b_2 \in H^3(\mathbb{Z}/p;\mathbb{Z}/p)$.
Although this was computed in \cite{murakami-ohtsuki-okada}, we give a proof based on a triangulation.
We represent $L(p,q)$ by an ordered $3$-cycle and give a labeling of $1$-simplices as indicated in Figure \ref{fig:lens}.
Here the triple $a$, $b$ and $c$ must satisfy $a^p = 1$ and $b=a^qbc$. 
Using multiplicative notation, the evaluating of $b_1b_2$ on this cycle is
\[
\sum_{i=0}^{p-1} b_1b_2([a|ia+b|-q a]) = \sum_{i=0}^{p-1} a \cdot b_2(ia+b, -q a) = -q a^2.
\] 
The second equality follows form the fact that $ia+b$ runs over $\mathbb{Z}/p$ and the equation (\ref{eq:2_cocycle_of_cyclic_group})
when $a$ is a unit, and also when $a$ is not a unit by a similar argument.
Therefore the Dijkgraaf-Witten invariant of $L(p,q)$ is equal to 
\[
\sum_{a=0}^{p-1} t^{-q a^2} \in \mathbb{Z}[t]/ (t^p-1) \cong \mathbb{Z}[\mathbb{Z}/p].
\]
Therefore this coincides with the shadow cocycle invariant up to a constant.

\begin{figure}
\begin{minipage}{270pt}
\begin{minipage}{120pt}

\vspace{20pt}
\input{lens_space.pstex_t}
\end{minipage}
\begin{minipage}{150pt}
\input{lens_developed.pstex_t}
\end{minipage}
\end{minipage}
\caption{A triangulation of $L(5,q)$ and a labeling where $a^5 = 1$ and $b=a^qbc$. }
\label{fig:lens}
\end{figure}


\begin{thebibliography}{99}
\bibitem[Ben]{benson}
D. J. Benson, 
{\it Representations and cohomology I : Basic representation theory of finite groups and associative algebras}, 
Cambridge Studies in Advanced Mathematics, 30. Cambridge University Press, Cambridge, 1991. xii+224 pp.

\bibitem[Bro]{brown}
K. Brown,
{\it Cohomology of groups},
Graduate Texts in Mathematics, 87. Springer-Verlag, New York-Berlin, 1982.

\bibitem[CJKLS]{CJKLS}
J. S. Carter, D. Jelsovsky, S. Kamada, L. Langford, M. Saito,
{\it Quandle cohomology and state-sum invariants of knotted curves and surfaces},
Trans. Amer. Math. Soc. 355 (2003), no. 10, 3947--3989.

\bibitem[Cla]{clauwens}
F. J. B. J. Clauwens,
{\it The algebra of rack and quandle cohomology},
arXiv:1004.4423

\bibitem[Eis1]{eisermann}
M. Eisermann,
{\it Homological characterization of the unknot},
Journal of Pure and Applied Algebra 177 (2003) 131--157.

\bibitem[Eis2]{eisermann_coloring}
M. Eisermann,
{\it Knot colouring polynomials},
Pacific J. Math. 231 (2007), no. 2, 305--336.

\bibitem[EG]{etingof-grana}
P. Etingof and M. Gra{\~n}a,
{\it On rack cohomology},
J. Pure Appl. Algebra 177 (2003), no. 1, 49--59. 

\bibitem[FRS]{FRS}
R. Fenn, C. Rourke and B. Sanderson,
{\it Trunks and classifying spaces},
Appl. Categ. Structures 3 (1995), no. 4, 321--356. 

\bibitem[HN]{hatakenaka-nosaka}
E. Hatakenaka and T. Nosaka,
{\it Some topological aspects of 4-fold symmetric quandle invariants of 3-manifolds},
preprint.

\bibitem[IK]{inoue-kabaya}
A. Inoue, Y. Kabaya,
{\it Quandle homology and complex volume},
arXiv:math/1012.2923.

\bibitem[Joy]{joyce}
D. Joyce, 
{\it A classifying invariant of knots, the knot quandle},
J. Pure Appl. Algebra 23 (1982), no. 1, 37--65.

\bibitem[Kam]{kamada}
S. Kamada,
{\it Quandles with good involutions, their homologies and knot invariants},
Intelligence of low dimensional topology 2006, 101--108, Ser. Knots Everything, 40, World Sci. Publ., Hackensack, NJ, 2007.

\bibitem[Moc]{mochizuki}
T. Mochizuki,
{\it Some calculations of cohomology groups of fnite Alexander quandles},
Journal of Pure and Applied Algebra 179 (2003) 287--330.

\bibitem[MOO]{murakami-ohtsuki-okada} 
H. Murakami, T. Ohtsuki, M. Okada, 
{\it Invariants of three-manifolds derived from linking matrices of framed links}, 
Osaka J. Math. 29 (1992), no. 3, 545--572.

\bibitem[Neu]{neumann}
W.D. Neumann, {\it Extended Bloch group and the Cheeger-Chern-Simons class}, Geom. 
Topol. 8 (2004), 413--474.

\bibitem[Nos1]{nosaka_alexander}
T. Nosaka, 
{\it On quandle homology groups of Alexander quandles of prime order}, 
preprint.

\bibitem[Nos2]{nosaka_surface}
T. Nosaka, 
{\it Quandle homotopy invariants of surface links},
arXiv:math/1011.6035.

\bibitem[Rol]{rolfsen}
D. Rolfsen, 
{\it Knots and links}, 
Mathematics Lecture Series, No. 7. Publish or Perish, Inc., Berkeley, Calif., 1976.

\bibitem[Wee]{weeks}
J. Weeks, 
{\it Computation of Hyperbolic Structures in Knot Theory},
Handbook of Knot Theory, 461--480.


\end{thebibliography}
\end{document}